\documentclass[11pt,a4paper]{article}

\usepackage{cmbright}
\usepackage{microtype}
\usepackage[T1]{fontenc}
\usepackage{lmodern}
\usepackage[utf8]{inputenc}

\usepackage{amsfonts}
\usepackage{amsmath}
\usepackage{amssymb}
\usepackage{amsthm}
\usepackage{cite}
\usepackage{hyperref}
\usepackage[usenames]{color}
\usepackage{tikz}
\usetikzlibrary{
	calc,
	intersections,
	through,
	backgrounds,
	arrows
}

\definecolor{darkgreen}{rgb}{0,0.4,0}
\definecolor{BrickRed}{rgb}{0.65,0.08,0}
\hypersetup{colorlinks=true,linkcolor=blue,citecolor=red,filecolor=BrickRed,urlcolor=darkgreen}

\theoremstyle{definition}
\newtheorem{theorem}{Theorem}[section]
\newtheorem{lemma}[theorem]{Lemma}
\newtheorem{proposition}[theorem]{Proposition}
\newtheorem{corollary}[theorem]{Corollary}
\newtheorem{definition}[theorem]{Definition}

\theoremstyle{remark}

\numberwithin{equation}{section}

\def\Nar[#1][#2]{
  N(#1,#2) 
}

\def\G{\mathcal{G}}

\def\P{{P}}

\def\e{\mathbf{e}}
\def\x{\mathbf{x}}
\def\y{\mathbf{y}}

\title{Brick Wall Excursions:\\ Combinatorial Interpretation of Random Flight Moments}

\usepackage[affil-it]{authblk}
\author[$a$]{Sergey \textsc{Kirgizov}}
\author[$a,b$]{Khaydar \textsc{Nurligareev}}
\author[$c$]{Michael \textsc{Wallner}}

\affil[$a$]{Université Bourgogne Europe, LIB UR 7534, F-21000 Dijon, France}
\affil[$b$]{Sorbonne Universit\'e, CNRS, LIP6 UMR 7606, F-75005 Paris, France}
\affil[$c$]{TU Wien, Austria}

\date{}

\begin{document}
\maketitle
\begin{center}
%{\LARGE Brick wall walks}

%\
%
%{\Large Sergey Kirgizov}\\
%LIB, Université de Bourgogne, Dijon, France
%\
%
%{\Large Khaydar Nurligareev}\\
%LIP6, Sorbonne Université, Paris, France
%\
%
%{\Large Michael Wallner}\\
%TU Wien, Vienna, Austria
\end{center}
\begin{abstract}
  We study the expected distance of short uniform random walks in arbitrary dimensions with unit steps in random directions.
  It is known that for dimensions $d=2$ and $d=4$, all the moments of an $m$-step walk are integer.
  While for $d=2$, the $n$th moment can be interpreted as the number of abelian squares of length $2n$ over an alphabet with $m$ letters, for $d=4$ no interpretation was known.

  The goal of this paper is to provide such an interpretation, both for $d=2$ and $d=4$, in terms of $2n$-step lattice paths in dimension $m-1$.
  Our construction relies on a bijection between Dyck paths with a prescribed number of peaks and words of a certain type.
  In addition, this bijection allows us to derive closed formulas for the number of lattice paths provided with certain statistics.
\end{abstract}

%  Basical use of the terms proposed by Michael:
%  - paths are the generic object;
%  - walks are unconstrained paths;
%  - excursions are paths returning to the origin and staying inside of some domain;
%  - bridges are paths that have the same ending and starting point, but are unconstrained. 
%  In our context, we call:
%  - walks the objects from the (initial) physical model;
%  - paths (or lattice paths) the combinatorial objects.

\section{Introduction}
\label{sec:introduction}

  In this paper, we consider short uniform random walks in higher dimensions, also known in the literature as \emph{random flights}.
  Following \cite{BorweinStraubVignat2016}, we fix two positive integers $d$ and $m$ and consider a uniform random walk in $\mathbb{R}^d$ that starts at the origin and consists of $m$ independent steps of length $1$, where the direction of each step is chosen uniformly at random.
  Let us denote by $p_m(\nu;x)$ the probability density function of the distance to the origin after $m$ such steps (also called the \emph{end-to-end distance}), and
  \begin{equation}\label{eq:moments-def}  
    W_m(\nu;s) = \int\limits_0^\infty x^s\, p_m(\nu;x)\, dx
  \end{equation}
  the associated moment functions.
  Here, the parameter $\nu$ is defined as
  \[
    \nu = \dfrac{d}{2} - 1
    \, .
  \]
  Obviously, for $m=0$ steps we have $W_0(\nu,s)=0$, since the walk stays at the origin, and for $m=1$ steps we have $W_1(\nu,s)=1$, since every possible step ends in the unit sphere.
  The interesting question now is: what happens for $m \geqslant 2$ steps?

  This question has a long history.
  One of its first appearances in the case where $d=2$ dates back to 1905 due to Pearson~\cite{Pearson1905} who was interested in the probability that the end-to-end distance belongs to an interval $(r,r + \delta r)$ in connection with mosquito migrations~\cite{Pearson1906}. 
  The first approach to solving the problem was proposed by Rayleigh~\cite{Rayleigh1905} and, more in depth, by Kluyver~\cite{Kluyver1906}.
  For $d=3$, the random walk under consideration was first extensively studied by Rayleigh~\cite{Rayleigh1919} a decade later.
  Pearson's initial question was also studied for any dimension in the case where the steps are not equal; see~\cite[§13.48]{Watson1966}.
  Beyond migration theory, short uniform random walks have found applications in a wide range of fields, such as, for example, the theory of elasticity for rubber-like materials~\cite{Treloar1946} and the analysis of polymer chain configurations~\cite{Yamakawa1971}.
  The reader can find a detailed review supplied with an extensive list of references in~\cite[Chapter~2]{Hughes1995}.

  The focus of many studies is on the distribution of the end-to-end distance, including its approximations for numerical purposes.
  The principal tools used for this goal are infinite integrals, Fourier transforms, and Bessel functions.
  In some cases, the cumbersome expressions given with the help of these tools simplify to nice and elegant forms.
  As an example of such behavior, we would like to mention Rayleigh’s theorem: the probability that the end-to-end distance is less than $1$ after $m$ steps is $1/(m+1)$; see~\cite{Bernardi2013}.
  Another fact worth mentioning is that odd dimensional results are resolvable in terms of elementary functions~\cite{Garcia-Pelayo2012,BorweinSinnamon2016}.

  The starting point of our study is also the result of such a simplification, which is, in a sense, a byproduct of the research conducted by Borwein, Straub, and their coauthors~\cite{BorweinNuyensStraubWan2011, BorweinStraubWanZudilin2012, BorweinStraubWan2013, BorweinStraubVignat2016}.
  The primary goal of the mentioned papers is to study the analytic and number theoretic behavior of short (five steps or less) uniform random walks in the plane and in higher dimensions, and in particular to determine explicit closed forms for the moments $W_m(\nu;s)$ defined by~\eqref{eq:moments-def}.
  Thus, among other results, it is shown that even moments have the following matrix representation (see \cite[Example~2.23]{BorweinStraubVignat2016}).
  For a given half-integer $\nu$, we define by 
  \[
    A(\nu) = \Big(A_{ij}(\nu)\Big)_{i,j\geqslant0}
  \]
  the infinite lower triangular matrix with the entries
  \begin{equation}
    \label{eq:A(nu)matrix}
    A_{ij}(\nu) = \binom{i}{j} \dfrac{(i+\nu)!\, \nu!}{(i-j+\nu)!\, (j+\nu)!}
    \, .
  \end{equation}
  Here, for non-integer values $x$, we understand factorials in terms of the Gamma function: $x!=\Gamma(x+1)$.
  Then, for two non-negative integers $m$ and $n$, the moment $W_{m+1}(\nu;2n)$ is given by the sum of the entries in the $n$th row of the matrix $A(\nu)^{m}$.
  In particular, this shows the surprising fact that for $\nu=0$ and $\nu=1$ (that is, for $d=2$ and $d=4$, respectively), the matrix entries and, consequently, the moments are integers.
  
  The fact that, for any given $m\in\mathbb{Z}_{>0}$, the elements of the sequences $\big(W_m(0;2n)\big)_{n\geqslant0}$ and $\big(W_m(1;2n)\big)_{n\geqslant0}$ are integers suggests that they could have a combinatorial interpretation.
  For $\nu=0$, it is indeed known that the moments $W_m(0;2n)$ count abelian squares of length $2n$ over an alphabet with $m$ letters (\emph{abelian squares} are strings $xy$ such that $y$ is a permutation of $x$, see \cite{RichmondShallit2009}).
  However, the moments $W_m(1;2n)$, to the best of our knowledge, had no combinatorial interpretation prior to this work.

  The main goal of this paper is to provide such an interpretation.
  We do it in a common manner, both for $\nu=0$ and $\nu=1$,
  in terms of lattice paths on some infinite graphs (lattices).
  With a few exceptions including interpretations of $W_2(0;2n)$, $W_3(0;2n)$, $W_4(0;2n)$, and $W_2(1;2n)$, this result seems to be novel.

  For $\nu=0$ $(d=2)$, the construction is as follows.
  
\begin{definition}\label{def:G_m(0)}
  Let $m>0$ and $(\e_1,\ldots,\e_m)$ be the standard basis in $\mathbb{R}^m$.
  Introduce an $(m+1)$-regular infinite graph $\G_m(0)$ by the following conditions.
  \begin{itemize}
    \item
	   The set of vertices of the graph $\G_m(0)$ is $\mathbb{Z}^{m}$.
    \item
	   Each vertex $\x=(x_1,\ldots,x_m)$ is connected to $(m+1)$ other vertices, namely:
      \begin{itemize}
        \item
          to the vertices $\x+\e_1$ and $\x-\e_1$,
        \item
          to the vertices $\x+\e_k$ in the case where the sum $(x_1+\dots+x_m)$ is even ($k=2,\ldots,m$),
        \item
          to the vertices $\x-\e_k$ in the case where the sum $(x_1+\dots+x_m)$ is odd ($k=2,\ldots,m$).
      \end{itemize}
  \end{itemize}
  The graphs $\G_{m}(0)$ for $m=1,2,3$ are shown in Figure~\ref{fig:G_m(0)}.
\end{definition}

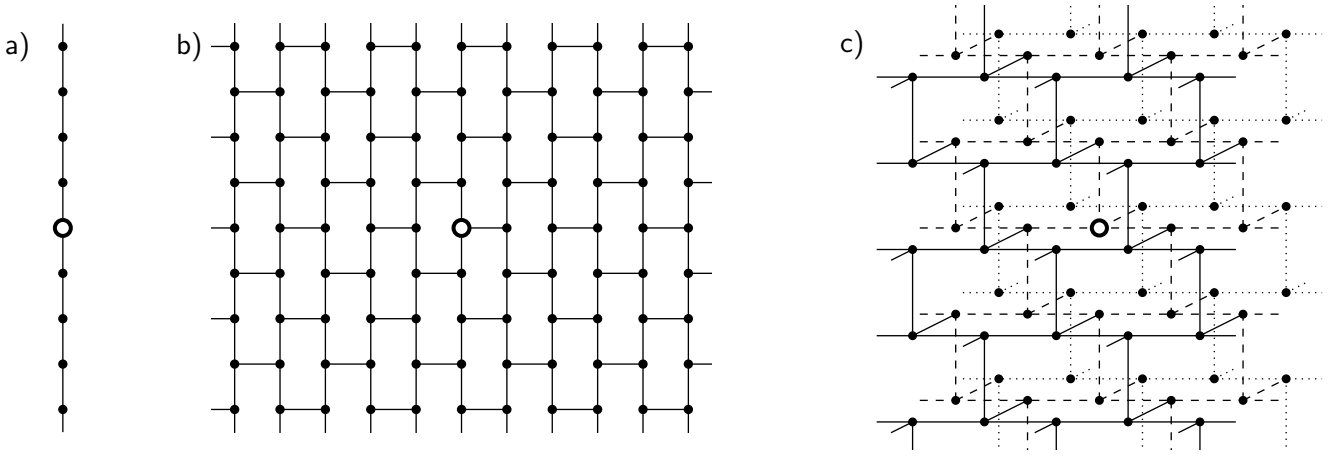
\begin{figure}[ht!]
\centering
\begin{tikzpicture}[line width=.5pt]
 \begin{scope}[scale=0.95, x={(1cm,0cm)}, y={(0.6cm,0.3cm)}, z={(0cm,1.2cm)}]
  \draw (-3.5,0.1,2.1) node {c)};
  \clip (-2.2,-1.5,-2.2) rectangle (2.2,1.5,2.2);
  \foreach \z in {-2,...,2}{
   \draw[dotted] (-2.5,1,\z) -- ++(5,0,0);
   \draw[dashed] (-2.5,0,\z) -- ++(5,0,0);
   \draw (-2.5,-1,\z) -- ++(5,0,0);
   \foreach \y in {-1,...,1}{
    \foreach \x in {-2,...,2}{
     \filldraw (\x,\y,\z) circle (1.5pt);
    }
   }
  }
  \foreach \z in {-2,0,2}{
   \foreach \x in {-2,0,2}{
    \draw (\x,-1,\z) -- ++(0,-0.5,0);
    \draw[dashed] (\x,0,\z) -- ++(0,1,0);
    \draw (\x,-1,\z) -- ++(0,0,-1);
    \draw[dashed] (\x,0,\z) -- ++(0,0,1);
    \draw[dotted] (\x,1,\z) -- ++(0,0,-1);
   }
   \foreach \x in {-1,1}{
    \draw (\x,-1,\z) -- ++(0,1,0);
    \draw[dotted] (\x,1,\z) -- ++(0,0.5,0);
    \draw (\x,-1,\z) -- ++(0,0,1);
    \draw[dashed] (\x,0,\z) -- ++(0,0,-1);
    \draw[dotted] (\x,1,\z) -- ++(0,0,1);
   }
  }
  \foreach \z in {-1,1}{
   \foreach \x in {-2,0,2}{
    \draw[dotted] (\x,1,\z) -- ++(0,0.5,0);
    \draw (\x,0,\z) -- ++(0,-1,0);
   }
   \foreach \x in {-1,1}{
    \draw (\x,-1,\z) -- ++(0,-0.5,0);
    \draw[dashed] (\x,0,\z) -- ++(0,1,0);
   }
  }
  \filldraw[white] (0,0,0) circle (3pt);
  \draw[line width=1.5pt] (0,0,0) circle (3pt);
 \end{scope}
 \begin{scope}[scale = 0.6, xshift = -400]
  \draw (-6,4) node {b)};
  \clip (-5.5,-4.5) rectangle (5.5,4.5);
  \foreach \x in {-5,...,5}
   \draw (\x,-5) -- ++(0,10);
  \foreach \x in {-3,...,2}{
   \foreach \y in {-2,...,2}{
    \draw (2*\x,2*\y) -- ++(1,0);
    \draw (2*\x+1,2*\y+1) -- ++(1,0);
    \filldraw (2*\x,2*\y) circle (2.5pt);
    \filldraw (2*\x+1,2*\y) circle (2.5pt);
    \filldraw (2*\x+1,2*\y+1) circle (2.5pt);
    \filldraw (2*\x+2,2*\y+1) circle (2.5pt);
   }
  }
  \filldraw[white] (0,0,0) circle (5pt);
  \draw[line width=1.5pt] (0,0,0) circle (5pt);
 \end{scope}
 \begin{scope}[scale = 0.6, xshift = -650]
  \draw (-1,4) node {a)};
  \draw (0,-4.5) -- ++(0,9);
  \foreach \y in {-4,...,4} \filldraw (0,\y) circle (2.5pt);
  \filldraw[white] (0,0,0) circle (5pt);
  \draw[line width=1.5pt] (0,0,0) circle (5pt);
 \end{scope}
\end{tikzpicture}
 \caption{The graph $\G_m(0)$:\quad a) for $m=1$,\quad b) for $m=2$,\quad c) for $m=3$.}
\label{fig:G_m(0)}
\end{figure}

  For $m=2$, this graph resembles a wall built of bricks.
  The same brick wall also appears in plane sections of the graph $\G_m(0)$ in the general case, namely, in sections parallel to the $x_1$-axis.
  That is why we informally refer to the paths on $\G_m(0)$ as the \emph{brick wall paths}.
  Note that the graph structure in the plane sections orthogonal to the $x_1$-axis resembles infinite $\Lambda$-shaped zigzags; see Figure~\ref{fig:G_m(0)}~c).

  In terms of brick wall paths, the combinatorial interpretation of the moments of the end-to-end distance of random flights in the plane is the following.

\begin{theorem}\label{thm:d=2_interpretation}
  For any two integers $m,n\in\mathbb{Z}_{>0}$, the moment $W_{m+1}(0;2n)$ is equal to the number of closed paths of length $2n$ on the graph $\G_m(0)$ which start and end at the origin.
\end{theorem}

  As we have already mentioned, the particular cases $m=1,2,3$ of Theorem~\ref{thm:d=2_interpretation} are known: the moments of $W_{m+1}(0;2n)$ are already interpreted as paths on certain lattices (see~\cite[Remark~4]{BorweinNuyensStraubWan2011}):
\begin{itemize}
  \item
    for $m=1$, these are paths on a line; see the sequence of central binomial coefficients \href{https://oeis.org/A000984}{\texttt{A000984}} in~\cite{oeis};
  \item
    for $m=2$, these are paths on the honeycomb lattice; see the sequence \href{https://oeis.org/A002893}{\texttt{A002893}} in~\cite{oeis};
  \item
    for $m=3$, these are paths on the diamond lattice; see the sequence of Domb numbers \href{https://oeis.org/A002895}{\texttt{A002895}} in~\cite{oeis}.
\end{itemize}
  At the same time, the combinatorial reason why these sequences appear in the context of moments $W_{m+1}(0;2n)$ is not explained in the works of Borwein, Straub, and their co-authors, nor, to the best of our knowledge, in any other articles.
  We fill this gap in Section~\ref{subsec:d=2}.
  Note that the graphs depicted in Figure~\ref{fig:G_m(0)}~b) and~c) are deformed versions of the honeycomb and diamond lattices, respectively, which are the same from the path counting point of view.
  Thus, Theorem~\ref{thm:d=2_interpretation} confirms the known counting results.

  For $\nu=1$ $(d=4)$, the construction is slightly different.

  \begin{definition}\label{def:G_m(1)}
  Let $m>0$ and $(\e_1,\ldots,\e_m)$ be the standard basis in $\mathbb{R}^m$.
  Introduce an infinite graph $\G_m(1)$ in the following way.
  \begin{itemize}
    \item
	   The set of vertices of the graph $\G_m(1)$ is $\mathbb{Z}_{\geqslant0}^m$.
    \item
	   Each vertex $\x=(x_1,\ldots,x_m)$ is connected to:
      \begin{itemize}
        \item
          the vertex $\x+\e_1$,
        \item
          the vertex $\x-\e_1$ in the case where $x_1>0$,
        \item
          the vertices $\x+\e_k$ in the case where the sum $(x_1+\ldots+x_k)$ is odd ($k=2,\ldots,m$),
        \item
          the vertices $\x-\e_k$ in the case where the sum $(x_1+\ldots+x_k)$ is even and $x_k>0$ ($k=2,\ldots,m$).
      \end{itemize}
  \end{itemize}
  The graphs $\G_{m}(1)$ for $m=1,2,3$ are shown in Figure~\ref{fig:G_m(1)}.
\end{definition}
  
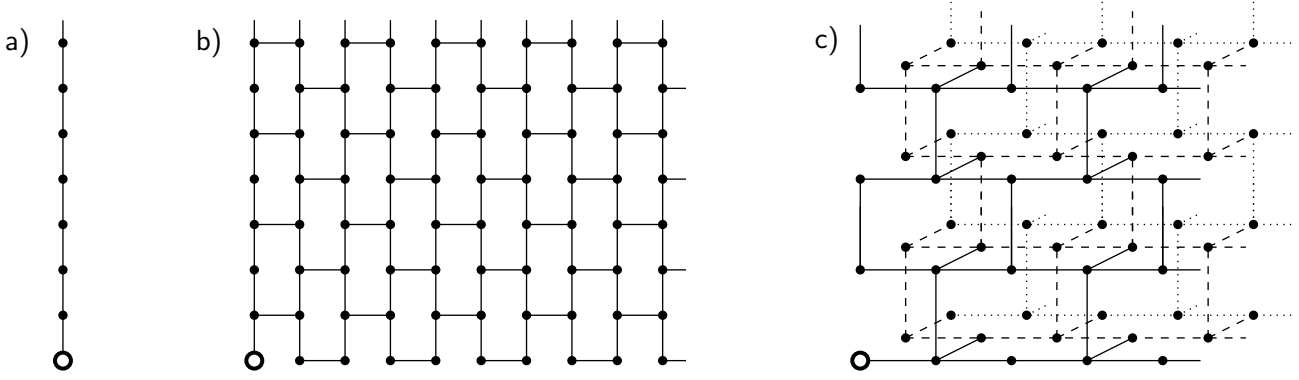
\begin{figure}[ht!]
\centering
\begin{tikzpicture}[line width=.5pt]
 \begin{scope}[scale=1, x={(1cm,0cm)}, y={(0.6cm,0.3cm)}, z={(0cm,1.2cm)}]
  \draw (-0.5,0.1,3.5) node {c)};
%  \clip (-0.2,-0.2,-0.2) rectangle (4.5,2.4,3.1);
  \foreach \z in {0,...,3}{
   \draw[dotted] (0,2,\z) -- ++(4.5,0,0);
   \draw[dashed] (0,1,\z) -- ++(4.5,0,0);
   \draw (0,0,\z) -- ++(4.5,0,0);
   \foreach \y in {0,...,2}{
    \foreach \x in {0,...,4}{
     \filldraw (\x,\y,\z) circle (1.5pt);
    }
   }
  }
  \foreach \z in {0,...,3}{
   \foreach \x in {0,2,4}{
    \draw[dashed] (\x,1,\z) -- ++(0,1,0);
   }
   \foreach \x in {1,3}{
    \draw (\x,0,\z) -- ++(0,1,0);
    \draw[dotted] (\x,2,\z) -- ++(0,0.5,0);
   }
  }
  \foreach \z in {0,2}{
   \foreach \x in {0,2,4}{
    \draw[dashed] (\x,1,\z) -- ++(0,0,1);
   }
   \foreach \x in {1,3}{
    \draw (\x,0,\z) -- ++(0,0,1);
    \draw[dotted] (\x,2,\z) -- ++(0,0,1);
   }
  }
  \foreach \z in {1}{
   \foreach \x in {0,2,4}{
    \draw (\x,0,\z) -- ++(0,0,1);
    \draw[dotted] (\x,2,\z) -- ++(0,0,1);
   }
   \foreach \x in {1,3}{
    \draw[dashed] (\x,1,\z) -- ++(0,0,1);
   }
  }
  \foreach \z in {1,3}{
   \foreach \x in {0,2,4}{
    \draw (\x,0,\z) -- ++(0,0,0.7);
    \draw[dotted] (\x,2,\z) -- ++(0,0,0.5);
   }
   \foreach \x in {1,3}{
    \draw[dashed] (\x,1,\z) -- ++(0,0,0.6);
   }
  }
  \filldraw[white] (0,0,0) circle (3pt);
  \draw[line width=1.5pt] (0,0,0) circle (3pt);
 \end{scope}
 \begin{scope}[scale = 0.6, xshift = -380]
  \draw (-1,7) node {b)};
  \clip (-0.5,-0.5) rectangle (9.5,7.5);
  \foreach \x in {0,...,9}
   \draw (\x,0) -- ++(0,8);
  \foreach \x in {0,...,4}{
   \foreach \y in {0,...,4}{
    \draw (2*\x,2*\y+1) -- ++(1,0);
    \draw (2*\x+1,2*\y) -- ++(1,0);
    \filldraw (2*\x,2*\y+1) circle (2.5pt);
    \filldraw (2*\x+1,2*\y+1) circle (2.5pt);
    \filldraw (2*\x+1,2*\y) circle (2.5pt);
    \filldraw (2*\x+2,2*\y) circle (2.5pt);
   }
  }
  \foreach \y in {0,...,3}{
   \filldraw (0,2*\y) circle (2.5pt);
  }
  \filldraw[white] (0,0,0) circle (5pt);
  \draw[line width=1.5pt] (0,0,0) circle (5pt);
 \end{scope}
 \begin{scope}[scale = 0.6, xshift = -500]
  \draw (-1,7) node {a)};
  \draw (0,0) -- ++(0,7.5);
  \foreach \y in {1,...,7} \filldraw (0,\y) circle (2.5pt);
  \filldraw[white] (0,0,0) circle (5pt);
  \draw[line width=1.5pt] (0,0,0) circle (5pt);
 \end{scope}
\end{tikzpicture}
 \caption{The graph $\G_m(1)$:\quad a) for $m=1$,\quad b) for $m=2$,\quad c) for $m=3$.}
\label{fig:G_m(1)}
\end{figure}
  
  Being defined on a hyperoctant, the graph $\G_m(1)$ can be seen as a subgraph of an $(m+1)$-regular infinite graph $\hat{\G}_m(1)$ whose set of vertices is $\mathbb{Z}^m$.
  In particular, the degree of any of its vertices $\x\in\mathbb{Z}_{>0}$ is $m+1$.
  The degrees of the boundary vertices of $\G_m(1)$ vary from $1$ to $m+1$, with the minimum reached only at the origin.
  
  Similarly to the graph $\G_m(0)$, the plane sections of the graph $\hat{\G}_m(1)$ parallel to the $x_1$-axis are brick walls,
  and so we informally refer to the paths on $\G_m(1)$ that start and end at the origin as \emph{brick wall excursions}.
  In contrast, the plane sections of $\hat{\G}_m(1)$ orthogonal to the $x_1$-axis resemble infinite $\Pi$-shaped zigzags; see Figure~\ref{fig:G_m(1)}~c).
  The difference comes from the last conditions of Definitions~\ref{def:G_m(0)} and~\ref{def:G_m(1)}.
  For example, given a vertex $\x\in\G_3(0)$, its adjacent vertices along the axes $x_2$ and $x_3$ are determined only by the parity of the coordinate sum $x_1+x_2+x_3$.
  At the same time, one of the adjacent vertices of $\x\in\G_3(1)$ depends on the parity of $x_1+x_2$, while another depends on $x_1+x_2+x_3$.
  More generally, the structure of the adjacent edges of a vertex $\x\in\G_m(0)$ is one of two types: $(\e_1, -\e_1, \e_2, \ldots, \e_m)$ or $(\e_1, -\e_1, -\e_2, \ldots, -\e_m)$, while the structure of the adjacent edges of a vertex $\x\in\hat{\G}_m(1)$ is one of $2^{m-1}$ types and is determined by the choice of signs in $(\e_1, -\e_1, \pm\e_2, \ldots, \pm\e_m)$. 
  That is why $\hat{\G}_1(1) = \G_1(0)$ and $\hat{\G}_2(1) = \G_2(0)$, but $\G_m(1) \neq \G_m(0)$ for $m>2$.

  The four dimensional moments of the end-to-end distance of random flights have the following combinatorial interpretation in terms of brick wall excursions.

\begin{theorem}\label{thm:d=4_interpretation}
  For any two integers $m,n\in\mathbb{Z}_{>0}$, the moment $W_{m+1}(1;2n)$ is equal to the number of closed paths of length $2(n+1)$ on the graph $\G_m(1)$ which start and end at the origin.
\end{theorem}

  The following interpretations for the moments $W_{m+1}(1;2n)$ were known prior to this work:
\begin{itemize}
  \item
    for $m=1$, these are Dyck paths; see the sequence of Catalan numbers \href{https://oeis.org/A000108}{\texttt{A000108}} in~\cite{oeis};
  \item
	for $m=2$, these are the row sums of the square of the Narayana triangle; see the sequence \href{https://oeis.org/A103370}{\texttt{A103370}} in~\cite{oeis}.
\end{itemize}
  For $m=3,4,5$, the corresponding sequences are also stocked in~\cite{oeis}; see the sequences \href{https://oeis.org/A253095}{\texttt{A253095}}, \href{https://oeis.org/A253096}{\texttt{A253096}} and~\href{https://oeis.org/A253095}{\texttt{A253097}}, respectively.

  Theorems~\ref{thm:d=2_interpretation} and~\ref{thm:d=4_interpretation} can also be considered in the context of Stieltjes moment sequences.
  Recall that $(a_n)_{\geqslant0}$ is a \emph{Stieltjes moment sequence} if there exists a measure $\mu$ on the set $\mathbb{R}_{\geqslant0}$ such that
  \begin{equation}\label{eq:Stielties-moments}
    a_n = \int x^nd\mu(x)
  \end{equation}
  for all $n\in\mathbb{Z}_{\geqslant0}$.
  Equivalently~\cite{Stieltjes1894, GantmakherKrein1937}, the \emph{Hankel matrices} $H^{(0)}_{\infty}(a)=\big(a_{i+j}\big)_{i,j\geqslant0}$ and $H^{(1)}_{\infty}(a)=\big(a_{i+j+1}\big)_{i,j\geqslant0}$ are both positive semidefinite, that is, they are symmetric and all their principal minors are non-negative.
  Another equivalent condition is that there exists a sequence of non-negative real numbers $(\alpha_n)_{n\geqslant0}$ such that the generating function for the sequence $(a_n)_{\geqslant0}$ satisfies
  \[
    \sum\limits_{n=0}^{\infty}a_nt^n
     = 
    \dfrac{\alpha_0}{1 - \dfrac{\alpha_1t}{1 - \dfrac{\alpha_2t}{1 - \ldots}}}
    \, .
  \]
  The initial problem stated by Stieltjes was to establish the necessary and sufficient conditions for a sequence $(a_n)_{\geqslant0}$ to be represented in the form~\eqref{eq:Stielties-moments}.
  Elvey Price and Guttmann~\cite[Theorem 4.3]{ElveyPriceGuttmann2019} formulated the following condition: if $\Gamma$ is a graph with a fixed base vertex $v_0$, such that each vertex in $\Gamma$ has degree at most $C\in\mathbb{Z}_{>0}$, and if $a_n$ is the number of paths of length $n$ in $\Gamma$ which start and end at $v_0$, then there exists a probability measure $\mu$ on $\mathbb{R}_{\geqslant0}$ such that for each $n\in\mathbb{Z}_{\geqslant0}$ the $n$th moment of $\mu$ is equal to $a_{2n}$
  (moreover, in this case the measure $\mu$ is unique and its support is contained in the interval $[0,C^2]$).
  In this context, our theorems are examples of the solution of the inverse problem: given a Stielties moment sequence $(a_n)_{\geqslant0}$, find a graph $\Gamma$ such that excursions on $\Gamma$ are counted by the sequence $(a_n)_{\geqslant0}$.

  For $\nu=0$, the structure of the graph $\G_m(0)$ described by Definition~\ref{def:G_m(0)} naturally comes from the problem.
  The idea is to consider the steps of a random flight as random unit complex variables.
  This allows us to represent the moments in terms of paths on a specific graph (Proposition~\ref{prop:d=2_pre-interpretation}), which, after an appropriate transformation, becomes our graph $\G_m(0)$.
  The matrix form~\eqref{eq:A(nu)matrix} also comes naturally in this case as a corollary (Corollary~\ref{cor:d=2_matrix_form}).

  For $\nu=1$, our proof is not direct.
  Instead, we rely on relation~\eqref{eq:A(nu)matrix} together with certain counting arguments.
  This leaves room for further explorations and the search for a potential natural explanation of Theorem~\ref{thm:d=4_interpretation}, probably in terms of quaternions.

  The structure of the paper is the following.
  We start an exposition with Section~\ref{sec:paths_and_words} recalling some useful facts related to paths and words.
  In particular, in the first part we recall Dyck paths, Motzkin paths, Catalan numbers, and Narayana numbers, and in the second part we show that a family of Motzkin paths is counted by Narayana numbers (Lemma~\ref{lem:pre_bijection}).
  Section~\ref{sec:interpretation} is devoted to our main theorems.
  Thus, in Section~\ref{subsec:d=2} we establish Theorem~\ref{thm:d=2_interpretation} using the basic properties of complex numbers and deduce relation~\eqref{eq:A(nu)matrix} for $\nu=0$, while in Section~\ref{subsec:d=4} we apply Lemma~\ref{lem:pre_bijection} together with relation~\eqref{eq:A(nu)matrix} to establish Theorem~\ref{thm:d=4_interpretation}.
  In Section~\ref{sec:bijection}, we rely on Lemma~\ref{lem:pre_bijection} to create a bijection between Motzkin paths of a special form and Dyck paths of a fixed number of peaks.
  We further apply Lemma~\ref{lem:pre_bijection} in Section~\ref{sec:brick_walks_enumeration} to enumerate paths on various cones of the brick wall lattice, i.e. on the graph $\G_2(0)$.
  Finally, we conclude the paper with Section~\ref{sec:conclusion} by discussing some directions for further research.

\section{Paths and words}
\label{sec:paths_and_words}

  In this section, we recall the notion of lattice paths and discuss some particular examples, such as Dyck paths and Motzkin paths.
  In the first part, we reveal the connections of bridges and Dyck paths with the matrix $A(\nu)$ through the peak statistic.
  We also recall related integer sequences, namely, the Catalan numbers, the Narayana numbers, and the Motzkin numbers.
  In the second part, we discuss one particular enumerative result that plays a~crucial role in the next sections of the paper.

\subsection{Preliminaries}
\label{subsec:preliminaries}
  
  As we mentioned in Section~\ref{sec:introduction}, the matrix entries~\eqref{eq:A(nu)matrix} are integers for $\nu=0$ and $\nu=1$.
  In both cases, they can be interpreted in terms of paths on the lattice $\mathbb{Z}^2$.
  These paths start at the origin and consist of several steps of two types: an \emph{up step} $U=(1,1)$ and a \emph{down step} $D=(1,-1)$.
  For the considered paths, the total numbers of up steps and down steps are the same, and so each path ends on the $x$-axis. If there are no other constraints, we call such paths \emph{bridges}; if our path is constrained to stay in a specific domain, we call it \emph{excursion}.
  
  A convenient way to work with a path is to represent it by a word over the alphabet $\{U,D\}$.
  In the case where the path is a bridge or an excursion, the number of occurrences of the letters $U$ and $D$ is equal.
  An important statistic that will come to the first plan in our investigation is a \emph{peak}, that is, the occurrence of two consecutive letters $UD$ in a word.

  Let us take a closer look at the case $\nu=0$.
  In this case, the entries of the matrix $A(\nu)$ become \emph{squares of binomial coefficients} (the sequence \href{https://oeis.org/A008459}{\texttt{A008459}} in~\cite{oeis}):
  \[
    A_{ij}(0) = \binom{i}{j}^2
    \, .
  \]
  This value $A_{ij}(0)$ counts bridges of length $2i$ with $j$ peaks (for example, $A_{20}(0)=1$, $A_{21}(0)=4$, $A_{22}(0)=1$, and the corresponding paths are shown in Figure~\ref{fig:A(0)_(ij)}).
  Indeed, to enumerate bridges of length $2i$ with $j$ peaks, we fix the $j$ occurrences of the peaks $UD$, and then treat them as separators between the $(i-j)$ remaining letters~$U$, which can thus be distributed in $\binom{(j+1)+(i-j)-1}{i-j}=\binom{i}{j}$ ways. 
  In the same way, independently of the letters~$U$, we treat the peaks as separators between the $(i-j)$ remaining letters $D$, which gives us the same number $\binom{i}{j}$ of possible choices for the letter positions.
  Since the positions of the letters $U$ and $D$ between two consecutive peaks are uniquely determined as $D\ldots DU\ldots U$, the number of paths under consideration is $\binom{i}{j}^2$.

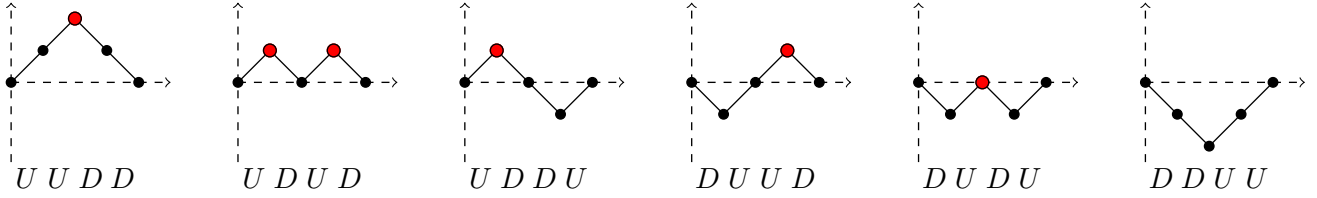
\begin{figure}[ht!]
\centering
\begin{tikzpicture}[line width=.5pt]
 \begin{scope}[scale=1.2, xshift = 5cm]
  \draw[->,dashed] (0pt,-25pt) -- (0,25pt);
  \draw[->,dashed] (0pt,0pt) -- (50pt,0pt);
  \foreach \x in {25pt,35pt} \draw (\x,-30pt) node {$U$};
  \foreach \x in {5pt,15pt} \draw (\x,-30pt) node {$D$};
%  \draw (23pt,-30pt) node {$DDUU$};
  \coordinate (a0) at (0pt,0pt);
  \coordinate (a1) at (10pt,-10pt);
  \coordinate (a2) at (20pt,-20pt);
  \coordinate (a3) at (30pt,-10pt);
  \coordinate (a4) at (40pt,0pt);
  \draw (a0) -- (a1) -- (a2) -- (a3) -- (a4);
  \foreach \p in {a0,a1,a2,a3,a4}
   \filldraw [black] (\p) circle (1.5pt);
 \end{scope}
 \begin{scope}[scale=1.2, xshift = 2.5cm]
  \draw[->,dashed] (0pt,-25pt) -- (0,25pt);
  \draw[->,dashed] (0pt,0pt) -- (50pt,0pt);
  \foreach \x in {15pt,35pt} \draw (\x,-30pt) node {$U$};
  \foreach \x in {5pt,25pt} \draw (\x,-30pt) node {$D$};
%  \draw (23pt,-30pt) node {$DUDU$};
  \coordinate (a0) at (0pt,0pt);
  \coordinate (a1) at (10pt,-10pt);
  \coordinate (a2) at (20pt,0pt);
  \coordinate (a3) at (30pt,-10pt);
  \coordinate (a4) at (40pt,0pt);
  \draw (a0) -- (a1) -- (a2) -- (a3) -- (a4);
  \foreach \p in {a0,a1,a3,a4}
   \filldraw [black] (\p) circle (1.5pt);
  \foreach \p in {a2}{
   \filldraw [red] (\p) circle (2pt);
   \draw (\p) circle (2pt);
  }
 \end{scope}
 \begin{scope}[scale=1.2, xshift = 0cm]
  \draw[->,dashed] (0pt,-25pt) -- (0,25pt);
  \draw[->,dashed] (0pt,0pt) -- (50pt,0pt);
  \foreach \x in {15pt,25pt} \draw (\x,-30pt) node {$U$};
  \foreach \x in {5pt,35pt} \draw (\x,-30pt) node {$D$};
%  \draw (23pt,-30pt) node {$DUUD$};
  \coordinate (a0) at (0pt,0pt);
  \coordinate (a1) at (10pt,-10pt);
  \coordinate (a2) at (20pt,0pt);
  \coordinate (a3) at (30pt,10pt);
  \coordinate (a4) at (40pt,0pt);
  \draw (a0) -- (a1) -- (a2) -- (a3) -- (a4);
  \foreach \p in {a0,a1,a2,a4}
   \filldraw [black] (\p) circle (1.5pt);
  \foreach \p in {a3}{
   \filldraw [red] (\p) circle (2pt);
   \draw (\p) circle (2pt);
  }
 \end{scope}
 \begin{scope}[scale=1.2, xshift = -2.5cm]
  \draw[->,dashed] (0pt,-25pt) -- (0,25pt);
  \draw[->,dashed] (0pt,0pt) -- (50pt,0pt);
  \foreach \x in {5pt,35pt} \draw (\x,-30pt) node {$U$};
  \foreach \x in {15pt,25pt} \draw (\x,-30pt) node {$D$};
%  \draw (23pt,-30pt) node {$UDDU$};
  \coordinate (a0) at (0pt,0pt);
  \coordinate (a1) at (10pt,10pt);
  \coordinate (a2) at (20pt,0pt);
  \coordinate (a3) at (30pt,-10pt);
  \coordinate (a4) at (40pt,0pt);
  \draw (a0) -- (a1) -- (a2) -- (a3) -- (a4);
  \foreach \p in {a0,a2,a3,a4}
   \filldraw [black] (\p) circle (1.5pt);
  \foreach \p in {a1}{
   \filldraw [red] (\p) circle (2pt);
   \draw (\p) circle (2pt);
  }
 \end{scope}
 \begin{scope}[scale=1.2, xshift = -5cm]
  \draw[->,dashed] (0pt,-25pt) -- (0,25pt);
  \draw[->,dashed] (0pt,0pt) -- (50pt,0pt);
  \foreach \x in {5pt,25pt} \draw (\x,-30pt) node {$U$};
  \foreach \x in {15pt,35pt} \draw (\x,-30pt) node {$D$};
%  \draw (23pt,-30pt) node {$UDUD$};
  \coordinate (a0) at (0pt,0pt);
  \coordinate (a1) at (10pt,10pt);
  \coordinate (a2) at (20pt,0pt);
  \coordinate (a3) at (30pt,10pt);
  \coordinate (a4) at (40pt,0pt);
  \draw (a0) -- (a1) -- (a2) -- (a3) -- (a4);
  \foreach \p in {a0,a2,a4}
   \filldraw [black] (\p) circle (1.5pt);
  \foreach \p in {a1,a3}{
   \filldraw [red] (\p) circle (2pt);
   \draw (\p) circle (2pt);
  }
 \end{scope}
 \begin{scope}[scale=1.2, xshift = -7.5cm]
  \draw[->,dashed] (0pt,-25pt) -- (0,25pt);
  \draw[->,dashed] (0pt,0pt) -- (50pt,0pt);
  \foreach \x in {5pt,15pt} \draw (\x,-30pt) node {$U$};
  \foreach \x in {25pt,35pt} \draw (\x,-30pt) node {$D$};
%  \draw (23pt,-30pt) node {$UUDD$};
  \coordinate (a0) at (0pt,0pt);
  \coordinate (a1) at (10pt,10pt);
  \coordinate (a2) at (20pt,20pt);
  \coordinate (a3) at (30pt,10pt);
  \coordinate (a4) at (40pt,0pt);
  \draw (a0) -- (a1) -- (a2) -- (a3) -- (a4);
  \foreach \p in {a0,a1,a2,a3,a4}
   \filldraw [black] (\p) circle (1.5pt);
  \foreach \p in {a2}{
   \filldraw [red] (\p) circle (2pt);
   \draw (\p) circle (2pt);
  }
 \end{scope}
\end{tikzpicture}
 \caption{Bridges of length $4$. Peaks are marked red.}
\label{fig:A(0)_(ij)}
\end{figure}

  Given an $i$th row of the matrix $A(0)$, the sum of its entries is the $i$th \emph{central binomial} coefficient (the sequence \href{https://oeis.org/A000984}{\texttt{A000984}} in~\cite{oeis}),
  \begin{equation}\label{eq:sum-A(0)}
    \sum\limits_{j=0}^i A_{ij}(0) = \sum\limits_{j=0}^i \binom{i}{j}^2 = \binom{2i}{i}
    \, ,
  \end{equation}
  which corresponds to the total number of bridges of length $2i$.

  \smallskip

  In the case where $\nu=1$, the matrix entries are \emph{Narayana numbers} (the sequence \href{https://oeis.org/A001263}{\texttt{A001263}} in~\cite{oeis}):
  \[
    A_{ij}(1) = \dfrac{1}{i+1}\binom{i+1}{j+1}\binom{i+1}{j} = \Nar[i+1][j+1]
    \, .
  \]
  Narayana numbers are related to \emph{Dyck paths}, that is, excursions described by the constraint that they remain in the quadrant $\mathbb{Z}_{\geqslant0}^2$.
  More precisely, given two positive integers $i$ and $j$, the Narayana number $\Nar[i][j]$ counts Dyck paths of length $2i$ with $j$ peaks (see~\cite{Deutsch1999}).
  For example, $\Nar[3][1]=1$, $\Nar[3][2]=3$, $\Nar[3][3]=1$, and the corresponding Dyck paths are shown in Figure~\ref{fig:A(1)_(ij)}.
  
\begin{figure}[ht!]
\centering
\begin{tikzpicture}[scale=1.2, line width=.5pt]
 \begin{scope}[xshift = -6cm]
  \draw[->,dashed] (0pt,0pt) -- (0,35pt);
  \draw[->,dashed] (0pt,0pt) -- (70pt,0pt);
  \foreach \x in {5pt,15pt,25pt} \draw (\x,-10pt) node {$U$};
  \foreach \x in {35pt,45pt,55pt} \draw (\x,-10pt) node {$D$};
%  \draw (30pt,-10pt) node {$UUUDDD$};
  \coordinate (a0) at (0pt,0pt);
  \coordinate (a1) at (10pt,10pt);
  \coordinate (a2) at (20pt,20pt);
  \coordinate (a3) at (30pt,30pt);
  \coordinate (a4) at (40pt,20pt);
  \coordinate (a5) at (50pt,10pt);
  \coordinate (a6) at (60pt,0pt);
  \draw (a0) -- (a1) -- (a2) -- (a3) -- (a4) -- (a5) -- (a6);
  \foreach \p in {a0,a1,a2,a4,a5,a6}
   \filldraw [black] (\p) circle (1.5pt);
  \foreach \p in {a3}{
   \filldraw [red] (\p) circle (2pt);
   \draw (\p) circle (2pt);
  }
 \end{scope}
 \begin{scope}[xshift = -3cm]
  \draw[->,dashed] (0pt,0pt) -- (0,35pt);
  \draw[->,dashed] (0pt,0pt) -- (70pt,0pt);
  \foreach \x in {5pt,15pt,35pt} \draw (\x,-10pt) node {$U$};
  \foreach \x in {25pt,45pt,55pt} \draw (\x,-10pt) node {$D$};
%  \draw (30pt,-10pt) node {$UUDUDD$};
  \coordinate (a0) at (0pt,0pt);
  \coordinate (a1) at (10pt,10pt);
  \coordinate (a2) at (20pt,20pt);
  \coordinate (a3) at (30pt,10pt);
  \coordinate (a4) at (40pt,20pt);
  \coordinate (a5) at (50pt,10pt);
  \coordinate (a6) at (60pt,0pt);
  \draw (a0) -- (a1) -- (a2) -- (a3) -- (a4) -- (a5) -- (a6);
  \foreach \p in {a0,a1,a3,a5,a6}
   \filldraw [black] (\p) circle (1.5pt);
  \foreach \p in {a2,a4}{
   \filldraw [red] (\p) circle (2pt);
   \draw (\p) circle (2pt);
  }
 \end{scope}
 \begin{scope}[xshift = 0cm]
  \draw[->,dashed] (0pt,0pt) -- (0,35pt);
  \draw[->,dashed] (0pt,0pt) -- (70pt,0pt);
  \foreach \x in {5pt,15pt,45pt} \draw (\x,-10pt) node {$U$};
  \foreach \x in {25pt,35pt,55pt} \draw (\x,-10pt) node {$D$};
%  \draw (30pt,-10pt) node {$UUDDUD$};
  \coordinate (a0) at (0pt,0pt);
  \coordinate (a1) at (10pt,10pt);
  \coordinate (a2) at (20pt,20pt);
  \coordinate (a3) at (30pt,10pt);
  \coordinate (a4) at (40pt,0pt);
  \coordinate (a5) at (50pt,10pt);
  \coordinate (a6) at (60pt,0pt);
  \draw (a0) -- (a1) -- (a2) -- (a3) -- (a4) -- (a5) -- (a6);
  \foreach \p in {a0,a1,a3,a4,a6}
   \filldraw [black] (\p) circle (1.5pt);
  \foreach \p in {a2,a5}{
   \filldraw [red] (\p) circle (2pt);
   \draw (\p) circle (2pt);
  }
 \end{scope}
 \begin{scope}[xshift = 3cm]
  \draw[->,dashed] (0pt,0pt) -- (0,35pt);
  \draw[->,dashed] (0pt,0pt) -- (70pt,0pt);
  \foreach \x in {5pt,25pt,35pt} \draw (\x,-10pt) node {$U$};
  \foreach \x in {15pt,45pt,55pt} \draw (\x,-10pt) node {$D$};
%  \draw (30pt,-10pt) node {$UDUUDD$};
  \coordinate (a0) at (0pt,0pt);
  \coordinate (a1) at (10pt,10pt);
  \coordinate (a2) at (20pt,0pt);
  \coordinate (a3) at (30pt,10pt);
  \coordinate (a4) at (40pt,20pt);
  \coordinate (a5) at (50pt,10pt);
  \coordinate (a6) at (60pt,0pt);
  \draw (a0) -- (a1) -- (a2) -- (a3) -- (a4) -- (a5) -- (a6);
  \foreach \p in {a0,a2,a3,a5,a6}
   \filldraw [black] (\p) circle (1.5pt);
  \foreach \p in {a1,a4}{
   \filldraw [red] (\p) circle (2pt);
   \draw (\p) circle (2pt);
  }
 \end{scope}
 \begin{scope}[xshift = 6cm]
  \draw[->,dashed] (0pt,0pt) -- (0,35pt);
  \draw[->,dashed] (0pt,0pt) -- (70pt,0pt);
  \foreach \x in {5pt,25pt,45pt} \draw (\x,-10pt) node {$U$};
  \foreach \x in {15pt,35pt,55pt} \draw (\x,-10pt) node {$D$};
%  \draw (30pt,-10pt) node {$UDUDUD$};
  \coordinate (a0) at (0pt,0pt);
  \coordinate (a1) at (10pt,10pt);
  \coordinate (a2) at (20pt,0pt);
  \coordinate (a3) at (30pt,10pt);
  \coordinate (a4) at (40pt,0pt);
  \coordinate (a5) at (50pt,10pt);
  \coordinate (a6) at (60pt,0pt);
  \draw (a0) -- (a1) -- (a2) -- (a3) -- (a4) -- (a5) -- (a6);
  \foreach \p in {a0,a2,a4,a6}
   \filldraw [black] (\p) circle (1.5pt);
  \foreach \p in {a1,a3,a5}{
   \filldraw [red] (\p) circle (2pt);
   \draw (\p) circle (2pt);
  }
 \end{scope}
\end{tikzpicture}
 \caption{Dyck paths of length $6$. Peaks are marked red.}
\label{fig:A(1)_(ij)}
\end{figure}
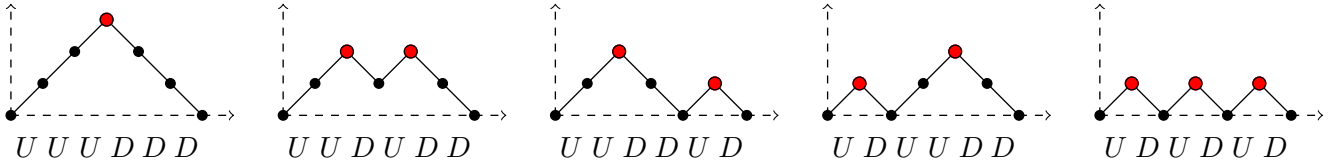

  Given an $i$th row of the matrix $A(1)$, the sum of its entries is equal to the $(i+1)$th \emph{Catalan number}:
  \begin{equation}\label{eq:sum-A(1)}
    \sum\limits_{j=0}^i A_{ij}(1) = \sum\limits_{j=0}^i \Nar[i+1][j+1] =\dfrac{1}{i+2}\binom{2(i+1)}{i+1} = C_{i+1}
    \, .
  \end{equation}
  The Catalan numbers count Dyck paths and are given by the sequence \href{https://oeis.org/A000108}{\texttt{A000108}} in~\cite{oeis}.

  \smallskip

  One more family of lattice paths that will be useful in our investigation is the one of Motzkin paths.
  Compared to Dyck paths, the set of allowed steps is enriched here with the \emph{horizontal step} $H=(1,0)$.
  More precisely, a~\emph{Motzkin path} of length $n$ is a path (excursion) consisting of $n$ steps (of types $U$, $D$, and $H$) that starts at the origin, ends on the $x$-axis, and stays in the quadrant $\mathbb{Z}^2_{\geqslant0}$.
  An example of a Motzkin path is shown in Figure~\ref{fig:Motzkin_path}.
  
\begin{figure}[ht!]
\centering
\begin{tikzpicture}[scale=1.2, line width=.5pt]
 \begin{scope}
  \draw[->,dashed] (0pt,0pt) -- (0,35pt);
  \draw[->,dashed] (0pt,0pt) -- (230pt,0pt);
%  \draw (110pt,-10pt) node {$UHUDHHUDHDUHHUHUHHDHDD$};
  \coordinate (a0) at (0pt,0pt);
  \coordinate (a1) at (10pt,10pt);
  \coordinate (a2) at (20pt,10pt);
  \coordinate (a3) at (30pt,20pt);
  \coordinate (a4) at (40pt,10pt);
  \coordinate (a5) at (50pt,10pt);
  \coordinate (a6) at (60pt,10pt);
  \coordinate (a7) at (70pt,20pt);
  \coordinate (a8) at (80pt,10pt);
  \coordinate (a9) at (90pt,10pt);
  \coordinate (a10) at (100pt,0pt);
  \coordinate (a11) at (110pt,10pt);
  \coordinate (a12) at (120pt,10pt);
  \coordinate (a13) at (130pt,10pt);
  \coordinate (a14) at (140pt,20pt);
  \coordinate (a15) at (150pt,20pt);
  \coordinate (a16) at (160pt,30pt);
  \coordinate (a17) at (170pt,30pt);
  \coordinate (a18) at (180pt,30pt);
  \coordinate (a19) at (190pt,20pt);
  \coordinate (a20) at (200pt,20pt);
  \coordinate (a21) at (210pt,10pt);
  \coordinate (a22) at (220pt,0pt);
  \draw (a0) -- (a1) -- (a2) -- (a3) -- (a4) -- (a5) -- (a6) -- (a7) -- (a8) -- (a9) -- (a10) -- (a11) -- (a12) -- (a13) -- (a14) -- (a15) -- (a16) -- (a17) -- (a18) -- (a19) -- (a20) -- (a21) -- (a22);
  \foreach \p in {a0,a1,a2,a3,a4,a5,a6,a7,a8,a9,a10,a11,a12,a13, a14,a15,a16,a17,a18,a19,a20,a21,a22}
   \filldraw [black] (\p) circle (1.5pt);
  \foreach \x in {5pt,25pt,65pt,105pt,135pt,155pt}{
   \draw (\x,-10pt) node {$U$};
  }
  \foreach \x in {35pt,75pt,95pt,185pt,205pt,215pt}{
   \draw (\x,-10pt) node {$D$};
  }
  \foreach \x in {15pt,45pt,55pt,85pt,115pt,125pt,145pt,165pt,175pt,195pt}{
   \draw (\x,-10pt) node {$H$};
  }
 \end{scope}
\end{tikzpicture}
 \caption{Motzkin path of length $22$.}
\label{fig:Motzkin_path}
\end{figure}
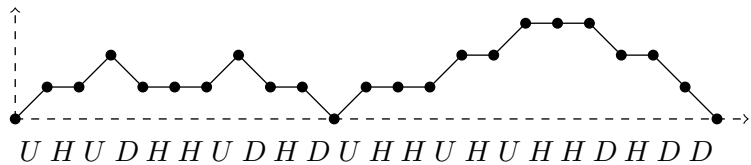

  The Motzkin paths are enumerated by \emph{Motzkin numbers} (the sequence \href{https://oeis.org/A001006}{\texttt{A001006}} in~\cite{oeis}):
  \[
    M_n = \sum\limits_{k=0}^{\lfloor\frac{n}{2}\rfloor}\binom{n}{2k}C_k
    \, .
  \]

\subsection{A family of Motzkin paths and Narayana numbers}
\label{subsec:Motzkin_vs_Narayana}

  Let $n$, $r$, and $\ell$ be three non-negative integers such that $r \geqslant \ell$.
  Consider the set $\P_{n,r,\ell}$ of words of length $n$ over the alphabet $\{U,D,H\}$ such that in each word of $\P_{n,r,\ell}$:
  \begin{itemize}
    \item
      the total number of $U$'s is $r$, and all of them are in odd positions, 
    \item
      the total number of $D$'s is $\ell$, and all of them are in even positions.
  \end{itemize}
  Denote also by $\hat{\P}_{n,r,\ell}$ the subset of $\P_{n,r,\ell}$ that additionally satisfies the following condition:
  \begin{itemize}
    \item
      in any prefix of a word of $\hat{\P}_{n,r,\ell}$, the number of $U$'s is not less than the number of $D$'s.
  \end{itemize}
  Clearly, the sets $\P_{n,r,\ell}$ and $\hat{\P}_{n,r,\ell}$ represent certain lattice paths.
  For example,
  \[
    \P_{4,1,1} = \{UDHH,UHHD,HHUD,HDUH\}
    \qquad\mbox{and}\qquad
    \hat{\P}_{4,1,1} = \{UDHH,UHHD,HHUD\}
    \, ,
  \]
  and the corresponding lattice paths are shown in Figure~\ref{fig:P_(4,1,1)}.
  In particular, the paths from the set $\hat{\P}_{n,r,\ell}$ stays in the quadrant $\mathbb{Z}_{\geqslant0}^2$, and for $r=\ell$, the set $\hat{\P}_{n,r,r}$ can be interpreted as a subset of the set of Motzkin paths of length $n$.

\begin{figure}[ht!]
\centering
\begin{tikzpicture}[line width=.5pt]
 \begin{scope}[scale=1.2, xshift = 0cm]
  \draw[->,dashed] (0pt,-15pt) -- (0,15pt);
  \draw[->,dashed] (0pt,0pt) -- (50pt,0pt);
  \draw (5pt,-20pt) node {$U$};
  \draw (15pt,-20pt) node {$D$};
  \draw (25pt,-20pt) node {$H$};
  \draw (35pt,-20pt) node {$H$};
%  \draw (23pt,-20pt) node {$UDHH$};
  \coordinate (a0) at (0pt,0pt);
  \coordinate (a1) at (10pt,10pt);
  \coordinate (a2) at (20pt,0pt);
  \coordinate (a3) at (30pt,0pt);
  \coordinate (a4) at (40pt,0pt);
  \draw (a0) -- (a1) -- (a2) -- (a3) -- (a4);
  \foreach \p in {a0,a1,a2,a3,a4}
   \filldraw [black] (\p) circle (1.5pt);
 \end{scope}
 \begin{scope}[scale=1.2, xshift = 2.5cm]
  \draw[->,dashed] (0pt,-15pt) -- (0,15pt);
  \draw[->,dashed] (0pt,0pt) -- (50pt,0pt);
  \draw (5pt,-20pt) node {$U$};
  \draw (15pt,-20pt) node {$H$};
  \draw (25pt,-20pt) node {$H$};
  \draw (35pt,-20pt) node {$D$};
%  \draw (23pt,-20pt) node {$UHHD$};
  \coordinate (a0) at (0pt,0pt);
  \coordinate (a1) at (10pt,10pt);
  \coordinate (a2) at (20pt,10pt);
  \coordinate (a3) at (30pt,10pt);
  \coordinate (a4) at (40pt,0pt);
  \draw (a0) -- (a1) -- (a2) -- (a3) -- (a4);
  \foreach \p in {a0,a1,a2,a3,a4}
   \filldraw [black] (\p) circle (1.5pt);
 \end{scope}
 \begin{scope}[scale=1.2, xshift = 5cm]
  \draw[->,dashed] (0pt,-15pt) -- (0,15pt);
  \draw[->,dashed] (0pt,0pt) -- (50pt,0pt);
  \draw (5pt,-20pt) node {$H$};
  \draw (15pt,-20pt) node {$H$};
  \draw (25pt,-20pt) node {$U$};
  \draw (35pt,-20pt) node {$D$};
%  \draw (23pt,-20pt) node {$HHUD$};
  \coordinate (a0) at (0pt,0pt);
  \coordinate (a1) at (10pt,0pt);
  \coordinate (a2) at (20pt,0pt);
  \coordinate (a3) at (30pt,10pt);
  \coordinate (a4) at (40pt,0pt);
  \draw (a0) -- (a1) -- (a2) -- (a3) -- (a4);
  \foreach \p in {a0,a1,a2,a3,a4}
   \filldraw [black] (\p) circle (1.5pt);
 \end{scope}
 \begin{scope}[scale=1.2, xshift = 7.5cm]
  \draw[->,dashed] (0pt,-15pt) -- (0,15pt);
  \draw[->,dashed] (0pt,0pt) -- (50pt,0pt);
  \draw (5pt,-20pt) node {$H$};
  \draw (15pt,-20pt) node {$D$};
  \draw (25pt,-20pt) node {$U$};
  \draw (35pt,-20pt) node {$H$};
%  \draw (23pt,-20pt) node {$HDUH$};
  \coordinate (a0) at (0pt,0pt);
  \coordinate (a1) at (10pt,0pt);
  \coordinate (a2) at (20pt,-10pt);
  \coordinate (a3) at (30pt,0pt);
  \coordinate (a4) at (40pt,0pt);
  \draw (a0) -- (a1) -- (a2) -- (a3) -- (a4);
  \foreach \p in {a0,a1,a2,a3,a4}
   \filldraw [black] (\p) circle (1.5pt);
 \end{scope}
\end{tikzpicture}
 \caption{Lattice paths from the set $\P_{4,1,1}$. The first three of them are Motzkin paths and form the set $\hat{\P}_{4,1,1}$.}
\label{fig:P_(4,1,1)}
\end{figure}
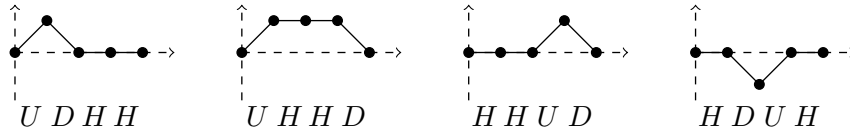

  Since, for the words of the sets $\P_{n,r,\ell}$, the choice of the positions of $U$'s and $D$'s is independent, we have the following relations for its cardinality:
  \begin{equation}\label{eq:D-relation}
    |\P_{2n,r,\ell}| = \binom{n}{r}\binom{n}{\ell}
    \qquad\mbox{and}\qquad
    |\P_{2n+1,r,\ell}| = \binom{n+1}{r}\binom{n}{\ell}
    \, .
  \end{equation}
  Calculating the cardinality of $\hat{\P}_{n,r,\ell}$ is less trivial.
  In the case where $r=\ell$, we use combinatorial ideas similar to the cycle lemma of Dvoretzky and Motzkin~\cite{DvoretzkyMotzkin1947} and Raney's lemma~\cite{Raney1960} (see also the proof of the fact that Narayana numbers $\Nar[n][k]$ enumerate Dyck paths of length $2n$ with $k$ peaks in~\cite[Exercise 6.36]{Stanley1999}).
  
\begin{lemma}\label{lem:pre_bijection}
  For any non-negative integers $n$ and $k$, we have
  \[
    |\hat{\P}_{2n,k,k}| = \Nar[n+1][k+1]
    \, .
  \]
\end{lemma}
\begin{proof}
  Consider the mapping $\varphi\colon\hat{P}_{2n,k,k}\to{P}_{2n+2,k+1,k}$ defined by adding the prefix $UH$ to the beginning of a~word, that is, $\varphi(w)=UHw$.
  Among the $n+1$ cyclic permutations of $\varphi(w)$ that belong to ${\P}_{2n+2,k+1,k}$, the word $\varphi(w)$ is characterized by the condition that all its non-empty prefixes contain more $U$'s than $D$'s.
  Moreover, for any word $w'\in{\P}_{2n+2,k+1,k}$, there is exactly one cyclic permutation of $w'$ with this property.
  That is why the mapping $\varphi$ can be interpreted as a bijection between $\hat{\P}_{2n,k,k}$ and the equivalence classes of ${\P}_{2n+2,k+1,k}$ with respect to cycle permutations.
  Since every class contains exactly $n+1$ elements, relation~\eqref{eq:D-relation} gives us
  \[
    |\hat{\P}_{2n,k,k}|
    = 
    \dfrac{\P_{2n+2,k+1,k}}{n+1}
    = 
    \dfrac{1}{n+1}\binom{n+1}{k+1}\binom{n+1}{k}
    = 
    \Nar[n+1][k+1]
    \, ,
  \]
  which concludes the proof.
\end{proof}

\section{Combinatorial interpretation of moment sequences}
\label{sec:interpretation}

  This section is devoted to our main result: a combinatorial interpretation of moments $W_m(\nu;2n)$ in terms of lattice paths.
  In the first part, we discuss the plane case ($d=2$); in the second part, we establish the result for the four dimensional case ($d=4$).

\subsection{Two dimensional case}
\label{subsec:d=2}

  Let us consider the plane case $d=2$, and therefore $\nu=0$.
  In this case, we identify $\mathbb{R}^2$ with the complex plane~$\mathbb{C}$,
  so that each step $A_k$ is a random complex number with modulus~$1$.
  Since $|A_k|^2=A_k\bar{A}_k=1$, the conjugate of each step coincides with its inverse: $\bar{A}_k=A_k^{-1}$. 

  For convenience, consider $m+1$ independent random variables $A_0,\ldots,A_m$ that are identically distributed in the unit circle.
  The above observation allows us to rewrite the $2n$th moment in the following way:
  \[
    W_{m+1}(0;2n)
     =
    \mathbb{E}\Big(
      |A_0^{{\color{white}1}} + \ldots + A_m^{{\color{white}1}}|^{2n}
    \Big)
%     =
%    \mathbb{E}\Big(
%      (A_0 + \ldots A_m)^{n}
%      (\bar{A}_0 + \ldots \bar{A}_m)^{n}
%    \Big)
%     =
%    \mathbb{E}\Big(
%      (A_0^{{\color{white}1}} + \ldots + A_m^{{\color{white}1}})^{n}
%      (A_0^{-1} + \ldots + A_m^{-1})^{n}
%    \Big)
     =
    \mathbb{E}\left(
      \Big(
        \big(A_0^{{\color{white}1}} + \ldots + A_m^{{\color{white}1}}\big)
        \big(A_0^{-1} + \ldots + A_m^{-1}\big)
      \Big)^{n}
    \right)
    \, .
  \]
  We can think of this expected value as a sum of $(m+1)^{2n}$ summands of the form
  \(
    \mathbb{E}\big(A_{i_1}^{{\color{white}1}}A_{j_1}^{-1}A_{i_2}^{{\color{white}1}}A_{j_2}^{-1}\ldots A_{i_n}^{{\color{white}1}}A_{j_n}^{-1}\big),
  \)
  where $0\leqslant i_k,j_k \leqslant m$ for all $k$ that satisfy $1\leqslant k \leqslant n$.
    
  Since $\mathbb{E}(A_k^s) = 0$ for any non-zero integer $s$, the independence of $A_k$ implies that the expected value of a~monomial $A=A_{i_1}^{{\color{white}1}}A_{j_1}^{-1}A_{i_2}^{{\color{white}1}}A_{j_2}^{-1}\ldots A_{i_n}^{{\color{white}1}}A_{j_n}^{-1}$ is non-zero if and only if all terms cancel out, i.e., $A=1$.
  Therefore, the moment $W_{m+1}(0;2n)$ is equal to the constant term in
  \(
    \Big(
      \big(A_0^{{\color{white}1}} + \ldots + A_m^{{\color{white}1}}\big)
      \big(A_0^{-1} + \ldots + A_m^{-1}\big)
    \Big)^{n}
  \):
%  $(A_0^{{\color{white}1}} + \ldots + A_m^{{\color{white}1}})^{n}(A_0^{-1} + \ldots + A_m^{-1})^{n}$:
  \[
    W_{m+1}(0;2n)
    =
    [A_0^0 \ldots A_m^0]
    \Big(
      \big(A_0^{{\color{white}1}} + \ldots + A_m^{{\color{white}1}}\big)
      \big(A_0^{-1} + \ldots + A_m^{-1}\big)
    \Big)^{n}
    \, .
  \]

  Now, let us give a lattice path interpretation of this constant term.
  To this end, consider the standard basis $(\e_0,\ldots,\e_m)$ in $\mathbb{R}^{m+1}$ and assign the unit vectors $\e_k$ and $-\e_k$, respectively, to the variables $A_k$ and their inverses $A_k^{-1}$, where $0 \leqslant k \leqslant m$.
  By doing so, we can interpret the monomial $A$ from above as the path $O,\e_{i_1},-\e_{j_1},\e_{i_2},-\e_{j_2},\ldots,\e_{i_n},-\e_{j_n}$ in $\mathbb{R}^{m+1}$, where $O=(0,\ldots,0)$ is the origin.
  In particular, $\mathbb{E}(A)=1$ if and only if the path is closed, that is, it returns to the origin (see Figure~\ref{fig:monomials_and_paths}).

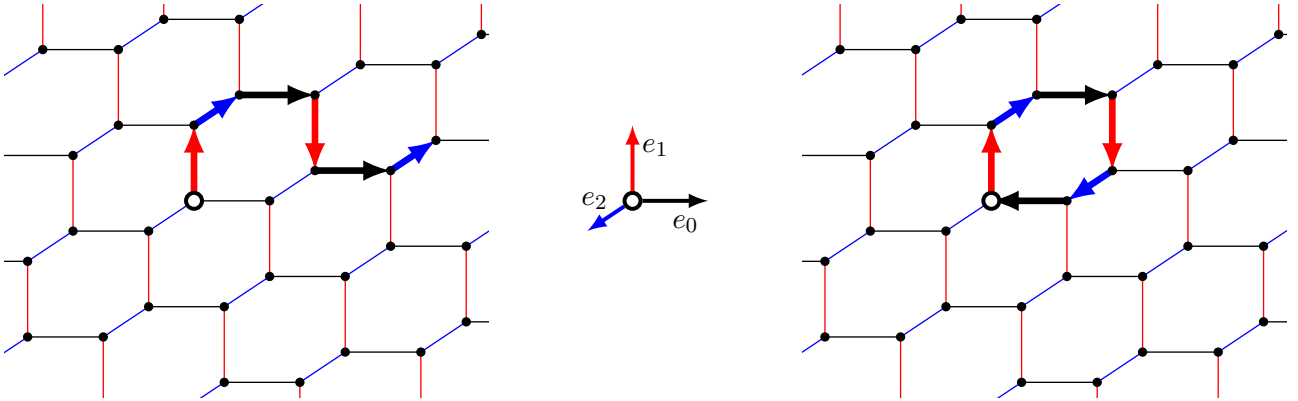
\begin{figure}[ht!]
\centering
\begin{tikzpicture}[>=latex, line width=.5pt, scale=1.0, x={(1cm,0cm)}, y={(-0.6cm,-0.4cm)}, z={(0cm,1cm)}]
 \begin{scope}[line width=1.5pt]
  \draw[->] (0,0,0) -- ++(1,0,0);
  \draw[->, blue] (0,0,0) -- ++(0,1,0);
  \draw[->, red] (0,0,0) -- ++(0,0,1);
  \filldraw[white] (0,0,0) circle (3pt);
  \draw[line width=1.5pt] (0,0,0) circle (3pt);
  \draw (0,-0.5,0.5) node {$e_1$};
  \draw (-0.5,0,0) node {$e_2$};
  \draw (1,0.5,-0.1) node {$e_0$};
 \end{scope}
 \begin{scope}[xshift = -165, yshift=0]
  \clip (-6.4,-6.5) rectangle (7.8,6.5);
  \draw[->, line width=2.5pt, red] (0,0,0) -- ++(0,0,1);
  \draw[->, line width=2.5pt, blue] (0,0,1) -- ++(0,-1,0);
  \draw[->, line width=2.5pt] (0,-1,1) -- ++(1,0,0);
  \draw[->, line width=2.5pt, red] (1,-1,1) -- ++(0,0,-1);
  \draw[->, line width=2.5pt] (1,-1,0) -- ++(1,0,0);
  \draw[->, line width=2.5pt, blue] (2,-1,0) -- ++(0,-1,0);
  \foreach \x in {-5,...,5}{
   \foreach \y in {-3,...,3}{
    \draw (\x,\y,-\x-\y) -- ++(1,0,0);
    \draw[blue] (\x,\y,-\x-\y) -- ++(0,1,0);
    \draw[red] (\x,\y,-\x-\y) -- ++(0,0,1);
   }
  }
  \foreach \x in {-5,...,5}{
   \foreach \y in {-3,...,3}{
    \filldraw (\x,\y,-\x-\y) circle (1.5pt);
    \filldraw (\x+1,\y,-\x-\y) circle (1.5pt);
   }
  }
  \filldraw[white] (0,0,0) circle (3pt);
  \draw[line width=1.5pt] (0,0,0) circle (3pt);
 \end{scope}
 \begin{scope}[xshift = 135, yshift=0]
  \clip (-6.4,-6.5) rectangle (7.8,6.5);
  \draw[->, line width=2.5pt, red] (0,0,0) -- ++(0,0,1);
  \draw[->, line width=2.5pt, blue] (0,0,1) -- ++(0,-1,0);
  \draw[->, line width=2.5pt] (0,-1,1) -- ++(1,0,0);
  \draw[->, line width=2.5pt, red] (1,-1,1) -- ++(0,0,-1);
  \draw[->, line width=2.5pt, blue] (1,-1,0) -- ++(0,1,0);
  \draw[->, line width=2.5pt] (1,0,0) -- ++(-1,0,0);
  \foreach \x in {-5,...,5}{
   \foreach \y in {-3,...,3}{
    \draw (\x,\y,-\x-\y) -- ++(1,0,0);
    \draw[blue] (\x,\y,-\x-\y) -- ++(0,1,0);
    \draw[red] (\x,\y,-\x-\y) -- ++(0,0,1);
   }
  }
  \foreach \x in {-5,...,5}{
   \foreach \y in {-3,...,3}{
    \filldraw (\x,\y,-\x-\y) circle (1.5pt);
    \filldraw (\x+1,\y,-\x-\y) circle (1.5pt);
   }
  }
  \filldraw[white] (0,0,0) circle (3pt);
  \draw[line width=1.5pt] (0,0,0) circle (3pt);
 \end{scope}
\end{tikzpicture}
 \caption{The paths on the lattice $\mathbb{Z}^3$ representing the monomials $A_1^{{\color{white}1}}A_2^{-1}A_0^{{\color{white}1}}A_1^{-1}A_0^{{\color{white}1}}A_2^{-1}$ (on the left) and $A_1^{{\color{white}1}}A_2^{-1}A_0^{{\color{white}1}}A_1^{-1}A_2^{{\color{white}1}}A_0^{-1}$ (on the right). Here, $m=2$ and $n=3$.}
\label{fig:monomials_and_paths}
\end{figure}

  The actual support of the set of paths corresponding to all possible monomials is the graph $(V_{m+1},E_{m+1})$ whose sets of vertices and edges are defined by
  \[
    V_{m+1}
      =
    \left\{
      (x_0,\ldots,x_m) \in \mathbb{Z}
      \,\,\colon\,\,\,
        \sum\limits_{k=0}^{m}x_k=0
        \quad\mbox{or}\quad
        \sum\limits_{k=0}^{m}x_k=1
    \right\}
  \]
  and
  \[
    E_{m+1}
      =
    \Big\{
      \x\y \,\colon\,\, 
        \x,\y \in V_{m+1},\,
        |\x-\y|=1
    \Big\}
    \, ,
  \]
  respectively. 
  This is a connected $(m+1)$-regular graph, which is infinite for $m>0$; see Figure~\ref{fig:d=2_pre-interpretation} for $m=0,1,2$.

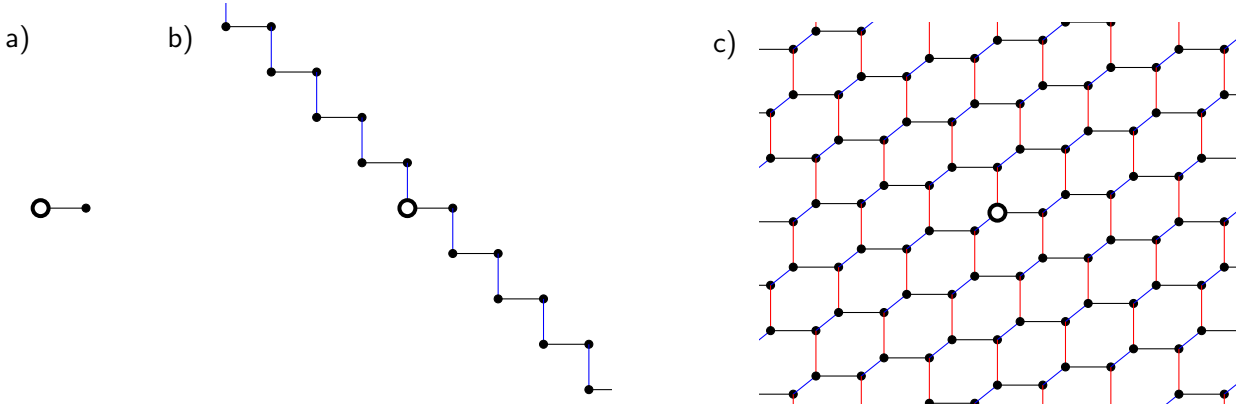
\begin{figure}[ht!]
\centering
\begin{tikzpicture}[scale = 0.6]
 \begin{scope}[xshift = -200]
  \draw (-0.5,3.7) node {a)};
  \draw (0,0) -- ++(1,0);
  \filldraw (1,0) circle (2.5pt);
  \filldraw[white] (0,0,0) circle (5pt);
  \draw[line width=1.5pt] (0,0,0) circle (5pt);
 \end{scope}
 \begin{scope}[xshift = 30]
  \draw (-5,3.7) node {b)};
  \clip (-4.5,-4.5) rectangle (4.5,4.5);
  \foreach \x in {-4,...,4}{
   \draw (\x,-\x) -- ++(1,0);
   \draw[blue] (\x,-\x) -- ++(0,1);
   \filldraw (\x,-\x) circle (2.5pt);
   \filldraw (\x+1,-\x) circle (2.5pt);
  }
  \filldraw[white] (0,0,0) circle (5pt);
  \draw[line width=1.5pt] (0,0,0) circle (5pt);
 \end{scope}
 \begin{scope}[xshift = 400, yshift=-0.1cm, x={(1cm,0cm)}, y={(-0.5cm,-0.4cm)}, z={(0cm,1cm)}]
  \draw (-4,4,5.3) node {c)};
  \clip (-10.5,-10.5) rectangle (10.5,10.5);
  \foreach \x in {-5,...,5}{
   \foreach \y in {-5,...,5}{
    \draw (\x,\y,-\x-\y) -- ++(1,0,0);
    \draw[blue] (\x,\y,-\x-\y) -- ++(0,1,0);
    \draw[red] (\x,\y,-\x-\y) -- ++(0,0,1);
    \filldraw (\x,\y,-\x-\y) circle (2.5pt);
    \filldraw (\x+1,\y,-\x-\y) circle (2.5pt);
   }
  }
  \filldraw[white] (0,0,0) circle (5pt);
  \draw[line width=1.5pt] (0,0,0) circle (5pt);
 \end{scope}
\end{tikzpicture}
 \caption{The moments $W_{m+1}(0;2n)$ count paths on the graph $(V_{m+1},E_{m+1})$ returning to the origin after $2n$~steps.
 The graph is given for\qquad a) $m=0$,\qquad b) $m=1$,\qquad c) $m=2$.}
\label{fig:d=2_pre-interpretation}
\end{figure}

  Thus, we just obtained the following combinatorial interpretation of the moments.

\begin{proposition}\label{prop:d=2_pre-interpretation}
  For any two integers $m,n\in\mathbb{Z}_{\ge0}$, the moment $W_{m+1}(0;2n)$ is equal to the number of closed paths of length $2n$ on the graph $(V_{m+1},E_{m+1})$ which start and end at the origin.
\end{proposition}

  Although the graph $(V_{m+1},E_{m+1})$ is defined in the space~$\mathbb{R}^{m+1}$, its actual structure is $m$-dimensional.
  In the case where $m>0$, the orthogonal projection of the graph $(V_{m+1},E_{m+1})$ onto the hyperplane 
  \[
    \Big\{
      (x_0,\ldots,x_m) \,:\, x_0+\ldots+x_m=0
    \Big\}
  \]
  gives us a regular lattice.
  For example, for $m=2$ we obtain the honeycomb lattice (Figure~\ref{fig:d=2,m=3}, on the left), while for $m=3$ the result is the so-called diamond lattice.
  From a counting point of view, it is more convenient for the lattice vertices to be integers, so we deform this regular lattice.
  Continuing our example for $m=2$, we get the lattice made up of bricks (Figure~\ref{fig:d=2,m=3}, on the right).

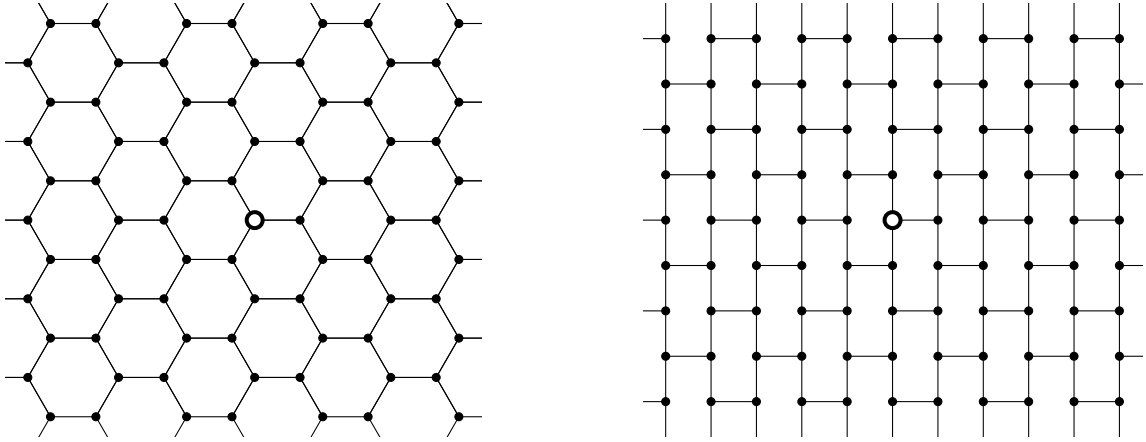
\begin{figure}[ht!]
\centering
\begin{tikzpicture}
 \def\h{5.2};
 \begin{scope}[scale = 0.1]
  \clip (-33,-5.5*\h) rectangle (30,5.5*\h);
  \foreach \x in {-2,...,1}{
   \foreach \y in {-3,...,2}{
    \draw (18*\x+9,2*\h*\y+\h) -- ++(6,0) -- ++(3,\h) -- ++(-3,\h) -- ++(-6,0) -- ++(-3,-\h) -- cycle;
    \draw (18*\x,2*\h*\y) -- ++(6,0) -- ++(3,\h) -- ++(-3,\h) -- ++(-6,0) -- ++(-3,-\h) -- cycle;
    \filldraw (18*\x,2*\h*\y) circle (15pt);
    \filldraw (18*\x+6,2*\h*\y) circle (15pt);
    \filldraw (18*\x+9,2*\h*\y+\h) circle (15pt);
    \filldraw (18*\x+15,2*\h*\y+\h) circle (15pt);
   }
  }
  \filldraw[white] (0,0,0) circle (30pt);
  \draw[line width=1.5pt] (0,0,0) circle (30pt);
 \end{scope}
 \begin{scope}[scale = 0.6, xshift = 400]
  \clip (-5.5,-5.5*\h/6) rectangle (5.5,5.5*\h/6);
  \foreach \x in {-5,...,5}
   \draw (\x,-5) -- ++(0,10);
  \foreach \x in {-3,...,2}{
   \foreach \y in {-2,...,2}{
    \draw (2*\x,2*\y) -- ++(1,0);
    \draw (2*\x+1,2*\y+1) -- ++(1,0);
    \filldraw (2*\x,2*\y) circle (2.5pt);
    \filldraw (2*\x+1,2*\y) circle (2.5pt);
    \filldraw (2*\x+1,2*\y+1) circle (2.5pt);
    \filldraw (2*\x+2,2*\y+1) circle (2.5pt);
   }
  }
  \filldraw[white] (0,0,0) circle (5pt);
  \draw[line width=1.5pt] (0,0,0) circle (5pt);
 \end{scope}
% \draw (4,0) node {$\rightsquigarrow$};
\end{tikzpicture}
 \caption{The moments $W_3(0;2n)$ count paths on the honeycomb lattice returning to the origin after $2n$~steps.
 From a counting point of view, the honeycomb lattice (on the left) is equivalent to the lattice made up of bricks (on the right).}
\label{fig:d=2,m=3}
\end{figure}

  The deformed lattice is exactly the graph $\G_m(0)$, which gives an explanation of Theorem~\ref{thm:d=2_interpretation} based on Proposition~\ref{prop:d=2_pre-interpretation}.
  In the following, we provide a rigorous justification for this claim.

\begin{proof}[Proof of Theorem~\ref{thm:d=2_interpretation}]
  Consider a transformation $\psi \colon \mathbb{Z}^{m+1} \to \mathbb{Z}^m$ defined by
  \[
    \psi(x_0,x_1,x_2,\ldots,x_m) = (x_1-x_0,x_2,\ldots,x_m)
    \, .
  \]
  The transformation $\psi$ provides a bijection between $V_{m+1}$ and $\mathbb{Z}^m$ whose inverse is given by
  \[
    \psi^{-1}(\tilde{x}_1,\tilde{x}_2,\ldots,\tilde{x}_m)
    =
    \left\{\begin{array}{ll}
      \big(-\tilde{s}/2,\tilde{x}_1-\tilde{s}/2,\tilde{x}_2,\ldots,\tilde{x}_m\big) & 
      \mbox{if } \tilde{s} \mbox{ is even}\\
      \big(-\tilde{s}/2,\tilde{x}_1+(1-\tilde{s})/2,\tilde{x}_2,\ldots,\tilde{x}_m\big) & 
      \mbox{if } \tilde{s} \mbox{ is odd} \, , \\
    \end{array}
    \right.
  \]
  where $\tilde{s} = \tilde{x}_1 + \ldots + \tilde{x}_m$.
  Extend $\psi$ to the set of edges $E_{m+1}$ by the following natural rule: if two vertices $\x,\y\in V_{m+1}$ are joined by an edge, then so are their images $\psi(\x)$ and $\psi(\y)$.
  Taking into account that the vertices $\x$ and $\y$ are joined by an edge if and only if $|\x-\y|=1$, we conclude that the image of the graph $(V_{m+1},E_{m+1})$ under the transformation $\psi$ is the graph $\G_m(0)$.
  Therefore, Theorem~\ref{thm:d=2_interpretation} follows due to Proposition~\ref{prop:d=2_pre-interpretation}.    
\end{proof}

\begin{corollary}\label{cor:d=2_matrix_form}
  For any two integers $m,n\in\mathbb{Z}_{\ge0}$, the moment $W_{m+1}(0;2n)$ is equal to the sum of the entries of the $n$th row of the matrix $M^m$, where $M=A(0)$ is defined by~\eqref{eq:A(nu)matrix}.
\end{corollary}
\begin{proof}
  For $m=0$, as well as for $n=0$, the statement holds trivially, since $W_{1}(0;2n)=W_{m+1}(0;0)=1$.
  For $m,n>0$, we apply Theorem~\ref{thm:d=2_interpretation}.
  To this end, let us encode a closed path on $\G_m(0)$ that starts and ends at the origin by a word over the alphabet
  \(
    \{U,D,R_2,L_2,\ldots,R_m,L_m\}.
  \)
  We use the following correspondence:
  \begin{itemize}
    \item
      the vector $\e_1$ is replaced by the letter $U$,
    \item
      the vector $-\e_1$ is replaced by the letter $D$,
    \item
      the vector $\e_k$ is replaced by the letter $R_k$ for $2 \leqslant k \leqslant m$,
    \item
      the vector $-\e_k$ is replaced by the letter $L_k$ for $2 \leqslant k \leqslant m$.
  \end{itemize}
  The structure of the graph $\G_m(0)$ imposes the following conditions on the considered words:
  \begin{itemize}
    \item 
      all $R_k$'s are in odd positions,
    \item 
      all $L_k$'s are in even positions,
    \item 
      $\#U = \#D$ and $\#R_k=\#L_k$ for $2 \leqslant k \leqslant m$.
  \end{itemize}
  Our goal is to establish the number of such words of length $2n$ for any positive integer $n$.

  For fixed numbers $\ell_1=\#U$ and $\ell_k=\#R_k$, $2 \leqslant k \leqslant m$, satisfying the condition $\ell_1+\ldots+\ell_m=n$, the number of words is equal to
  \[
    \binom{n}{\ell_1,\ldots,\ell_m}^2
    \binom{2\ell_1}{\ell_1}
    =
    \binom{n}{\ell_m}^2
    \binom{n-\ell_m}{\ell_{m-1}}^2
    \ldots
    \binom{\ell_1+\ell_2+\ell_3}{\ell_3}^2
    \binom{\ell_1+\ell_2}{\ell_2}^2
    \binom{2\ell_1}{\ell_1}
  \]
  Indeed, first we distribute the $U$'s and $R_k$'s across the $n$ odd positions, then we distribute the $D$'s and  $L_k$'s across the $n$ even positions, and finally we re-distribute $U$'s and $D$'s across the $2\ell_1$ positions that remain.
  Hence, we have
  \[
    W_{m+1}(0;2n)
    =
    \sum\limits_{\ell_1+\ldots+\ell_m=n}
    \binom{n}{\ell_1,\ldots,\ell_m}^2
    \binom{2\ell_1}{\ell_1}
    \, .
  \]
  Denoting $n_k=\ell_1+\ldots+\ell_k$ for $2 \leqslant k \leqslant m$ and taking into account that 
  \(
    \binom{2\ell_1}{\ell_1} = \binom{\ell_1}{0}^2+\ldots+\binom{\ell_1}{\ell_1}^2,
  \)
  as shown in relation~\eqref{eq:sum-A(0)}, we can rewrite the above expression in the form
  \[
    W_{m+1}(0;2n)
    =
    \sum\limits_{0 \leqslant n_1 \leqslant \ldots \leqslant n_m=n}
    \binom{2\ell_1}{\ell_1}
    \prod\limits_{k=2}^m
    \binom{n_k}{n_{k-1}}^2
    =
    \sum\limits_{0 \leqslant n_0 \leqslant \ldots \leqslant n_m=n}
    \prod\limits_{k=1}^m
    \binom{n_k}{n_{k-1}}^2
    \, .
  \]
  Finally, since $\binom{i}{j}=0$ for $i<j$, we obtain a matrix form expression:
  \[
    W_{m+1}(0;2n)
    =
    \sum\limits_{n_0=0}^n
    \ldots
    \sum\limits_{n_{m-1}=0}^n
    \prod\limits_{k=1}^m
    \binom{n_k}{n_{k-1}}^2
    =
    \sum\limits_{n_0=0}^n
    \ldots
    \sum\limits_{n_{m-1}=0}^n
    M_{n_mn_{m-1}} \ldots M_{n_1n_0}
    =
    \sum\limits_{n_0=0}^n
    [M^m]_{nn_0}
    \, .
  \]
  This completes the proof.
\end{proof}

\subsection{Four dimensional case}
\label{subsec:d=4}

\begin{proof}[Proof of Theorem~\ref{thm:d=4_interpretation}]
  The main idea of the proof is to verify that the number of excursions of length $2(n+1)$ on the graph $\G_m(1)$ is equal to the sum of the entries of the $n$th row of the matrix $N^m$, where $N=A(1)$ is defined by~\eqref{eq:A(nu)matrix}, and hence, coincides with the moment $W_{m+1}(1;2n)$.
  This is to be done by induction on $m$ with the help of Lemma~\ref{lem:pre_bijection}.

  Let us check the base case.
  For $m=1$, the graph $\G_1(1)$ is a half-line, and the number of excursions on it is counted by the Catalan numbers:
  \[
    C_{n+1} = \sum_{k=0}^n \Nar[n+1][k+1] = \sum_{k=0}^n N_{nk} = W_2(1;2n)
    \, ;
  \]
  see relation~\eqref{eq:sum-A(1)}.
  For $m=2$, we use Lemma~\ref{lem:pre_bijection}.
  In this case, the graph $\G_2(1)$ is a quarter-plane shown in Figure~\ref{fig:G_m(1)}~b), and the paths on $\G_2(1)$ are encoded by words over the alphabet $\{U,D,R,L\}$.
  Since the first and the last step of an excursion of length $2(n+1)$ are uniquely defined, it is represented by a word $w=Uw'D$ for some word $w'$ of length $2n$.
  To enumerate all possible words $w'$, we first fix the number $k$ of $U$'s (and $D$'s).
  Hence, the number of $R$'s (and $L$'s) is equal to $(n-k)$,
  and the number of ways to distribute $R$'s and $L$'s in $w'$, due to Lemma~\ref{lem:pre_bijection}, is $\Nar[n+1][n-k+1]=\Nar[n+1][k+1]$.
  After that, we fill the remaining positions in the word with $U$'s and $D$'s, which can be done in $C_{n+1}$ ways.
  Finally, we take the sum over all $k\leqslant n$ and rewrite the formula as
  \[
    \sum\limits_{k=0}^n
      C_{k+1}\Nar[n+1][k+1]
      =
%      \sum_{k=0}^n\sum_{\ell=0}^k \Nar[n+1][k+1] \Nar[k+1][\ell+1]
      \sum_{k=0}^n\sum_{\ell=0}^k N_{nk} N_{k\ell}
      =
      \sum_{\ell=0}^n\sum_{k=0}^n N_{nk} N_{k\ell}
      =
      \sum_{\ell=0}^n [N^2]_{n\ell}
      =
 	   W_3(1;2n)
    \, .
  \]
  Here, we use relation~\eqref{eq:sum-A(1)} and the fact that the matrix $A(\nu)$ is lower triangular.

  Now, proceed with the induction step.
  In order to apply Lemma~\ref{lem:pre_bijection}, similarly to Corollary~\ref{cor:d=2_matrix_form} we interpret a path of length $2(n+1)$ on the graph $\G_m(1)$ as a word $w$ of the same length over the alphabet $\{U,D,R_2,L_2,\ldots,R_m,L_m\}$.
%  The letters $U$ and $D$ represent two opposite steps taken along the $x_1$-axis, while $R_k$ and $L_k$ correspond to steps along the $x_k$-axis that, respectively, increase and decrease the $x_k$-coordinate (here, $2\leqslant k \leqslant m$).
  Again, as for $m=2$, we have $w=Uw'D$ for some word $w'$ of length $2n$.
  Due to the structure of the graph~$\G_m(1)$, the word $w'$ satisfies the following four conditions.
  \begin{enumerate}
    \item
      The number of occurrences of $U$'s and $D$'s in $w'$ is equal, as well as those of $R_k$'s and $L_k$'s for all $k$.
    \item 
      For any $k$, after erasing all letters $R_{k+1},L_{k+1},\ldots,R_m,L_m$ in the word $w'$, the letters $R_k$ occur in odd positions in the remaining word.
    \item 
      For any $k$, after erasing all letters $R_{k+1},L_{k+1},\ldots,R_m,L_m$ in the word $w'$, the letters $L_k$ occur in even positions in the remaining word.
    \item 
      In each prefix of $w'$, for every $2\leqslant k \leqslant m$, the number of $R_k$'s is not less than the number of $L_k$'s.
  \end{enumerate}
  The first of these conditions guarantees that the considered path returns to the origin.
  The last condition certifies that the path remains in the hyperoctant $\mathbb{Z}^m_{\geqslant0}$, so it is indeed an excursion.
  The other two conditions come from the inner structure of the graph $\G_m(1)$.

  Let $k$ be the number of occurrences of the letter $R_m$ in the word $w'$.
  Then, according to Lemma~\ref{lem:pre_bijection}, the number of ways to choose positions for $R_m$'s and $L_m$'s is $\Nar[n+1][k+1]=\Nar[n+1][n-k+1]=N_{n(n-k)}$.
  On the other hand, it follows from the induction hypothesis that the number of ways to fill the rest of the word with other letters is $\sum_{\ell=0}^{n-k}[N^{m-1}]_{(n-k)\ell}$.
  Therefore, since the matrix $N^{m-1}$ is lower triangular, the total number of possible words $w'$ is equal to
  \[
    \sum_{k=0}^n N_{n(n-k)} \sum_{\ell=0}^{n-k} [N^{m-1}]_{(n-k)\ell}
    =
    \sum_{\ell=0}^{n}\sum_{k=0}^n N_{n(n-k)} [N^{m-1}]_{(n-k)\ell}
    =
    \sum_{\ell=0}^{n} [N^{m}]_{n\ell}
    =
 	  W_{m+1}(1;2n)
    \, ,
  \]
  which completes the proof.
\end{proof}

\section{Bijection}
\label{sec:bijection}

  Lemma~\ref{lem:pre_bijection} suggests that there should be a bijection between Motzkin paths of a certain kind and Dyck paths with a predefined number of peaks.
  This section is devoted to an iterative algorithm that allows us to obtain such a bijection presented by the following proposition.

\begin{proposition}\label{prop:bijection}
  For any non-negative integers $n$ and $k$, there exists a bijection between the set $\hat{\P}_{2n,k,k}$ of Motzkin paths (of length $2n$ whose $k$ up steps are in odd positions and $k$ down steps are in even positions) and the set of Dyck paths of size $2(n+1)$ with $k+1$ peaks.
\end{proposition}
\begin{proof}
  The idea is to pass from a Motzkin path to a Dyck path by inserting two additional steps ($U$ and~$D$) and iteratively transforming the path from bottom to top.
  Thus, the algorithm that realizes this idea consists of two stages.
  The first stage is preliminary: we identify the first horizontal step $H$ that belongs to the $x$-axis and insert a step~$U$ just in front of it.
  At the same time, we insert a step~$D$ in the end of the path; see Figure~\ref{fig:bijection_algorithm}~b).
  If there are no horizontal steps on the $x$-axis, then we insert both steps $U$ and $D$ in the end of the path.
  In any case, the result of the preliminary stage is a Motzkin path of length $2(n+1)$ that contains no horizontal step on the $x$-axis.

  The second stage of the algorithm is iterative.
  We decompose our Motzkin path into a sequence of \emph{prime paths}, that is, parts of the form $U \alpha D$ where $\alpha$ is a smaller Motzkin path.
  In other words, a prime path touches the $x$-axis only twice: in the beginning and in the end.
  Each prime path $U \alpha D$ is then treated separately in the following way.
  If $\alpha$ consists only of horizontal steps, then we replace the first half of them by up steps, and the second half of them by down steps:
  \[
    U\underbrace{H \ldots H}_{2\ell}D
    \quad\rightsquigarrow\quad
    \underbrace{U \ldots U}_{\ell+1}\underbrace{D \ldots D}_{\ell+1}
    \, .
  \]
  Otherwise, the prime path under consideration can be represented in the form
  \[
    U \alpha D =
    U\underbrace{H \ldots H}_{x}
    U \beta_1 D \underbrace{H \ldots H}_{2\ell_1}
    U \beta_2 D \ldots U \beta_{s-1} D
    \underbrace{H \ldots H}_{2\ell_{s-1}} U \beta_s D
    \underbrace{H \ldots H}_{y}D
    \, ,
  \]
  where $\beta_i$ are some Motzkin paths.
  Note that, due to the fact that the initial path belongs to $\hat{\P}_{2n,k,k}$, in the typical case the values of $x$ and $y$ are odd.
  The only exception concerns the prime path whose first and last letters have been inserted in the preliminary stage of the algorithm: in this case, both $x$ and $y$ are even.
  In any case, we define two integers $u$ and $v$ to be
  \[
    u = \left\{\begin{array}{ll}
      x-1 & \mbox{if } x\leqslant y\\
      y & \mbox{if } x>y 
    \end{array}\right.
    \qquad\mbox{and}\qquad
    v = \left\{\begin{array}{ll}
      \frac{y-x}{2} + 1 & \mbox{if } x\leqslant y\\
      \frac{x-y}{2} & \mbox{if } x>y \, ,
    \end{array}\right.
  \]
  so that $2(u+v) = x+y$.
  This allows us to rearrange the occurrences of $H$ in our prime path into
  \[
    U\underbrace{H \ldots H}_{u+(\ell_1+\ldots+\ell_{s-1})+2v}
    U \beta_1 D \underbrace{H \ldots H}_{\ell_1}
    U \beta_2 D \ldots U \beta_{s-1} D
    \underbrace{H \ldots H}_{\ell_{s-1}} U \beta_s D
    \underbrace{H \ldots H}_{u}D
    \, ,
  \]
  and then replace the first $u+(\ell_1+\ldots\ell_{s-1})+v$ horizontal steps $H$ by up steps $U$ and the remaining horizontal steps $H$ by down steps $D$:
  \[
%    \underbrace{U \ldots U}_{u+(\ell_1+\ldots+\ell_{s-1})+v+1}
    \underbrace{U \ldots U}_{u+1} \underbrace{U \ldots U}_{\ell_1+\ldots+\ell_{s-1}}
    \underbrace{U \ldots U}_{v} \underbrace{D \ldots D}_{v}
    U \beta_1 D \underbrace{D \ldots D}_{\ell_1}
    U \beta_2 D \ldots U \beta_{s-1} D
    \underbrace{D \ldots D}_{\ell_{s-1}} U \beta_s D
    \underbrace{D \ldots D}_{u+1}
    \, .
  \]
  What remains is to apply the same iteration to the prime paths $U \beta_i D$ and so on; see Figure~\ref{fig:bijection_algorithm}.
  
\begin{figure}[ht!]
\centering
\begin{tikzpicture}[scale=1, line width=.5pt]
 \begin{scope}[yshift=0]
  \draw[->,dashed] (0pt,0pt) -- (0,35pt);
  \draw[->,dashed] (0pt,0pt) -- (230pt,0pt);
  \draw (10pt,30pt) node {a)};
  \coordinate (a0) at (0pt,0pt);
  \coordinate (a1) at (10pt,10pt);
  \coordinate (a2) at (20pt,10pt);
  \coordinate (a3) at (30pt,20pt);
  \coordinate (a4) at (40pt,10pt);
  \coordinate (a5) at (50pt,10pt);
  \coordinate (a6) at (60pt,10pt);
  \coordinate (a7) at (70pt,20pt);
  \coordinate (a8) at (80pt,10pt);
  \coordinate (a9) at (90pt,10pt);
  \coordinate (a10) at (100pt,0pt);
  \coordinate (a11) at (110pt,0pt);
  \coordinate (a12) at (120pt,0pt);
  \coordinate (a13) at (130pt,10pt);
  \coordinate (a14) at (140pt,10pt);
  \coordinate (a15) at (150pt,20pt);
  \coordinate (a16) at (160pt,20pt);
  \coordinate (a17) at (170pt,20pt);
  \coordinate (a18) at (180pt,10pt);
  \coordinate (a19) at (190pt,10pt);
  \coordinate (a20) at (200pt,0pt);
  \draw (a0) -- (a1) -- (a2) -- (a3) -- (a4) -- (a5) -- (a6) -- (a7) -- (a8) -- (a9) -- (a10) -- (a11) -- (a12) -- (a13) -- (a14) -- (a15) -- (a16) -- (a17) -- (a18) -- (a19) -- (a20);
  \foreach \p in {a0,a1,a2,a3,a4,a5,a6,a7,a8,a9,a10,a11,a12,a13, a14,a15,a16,a17,a18,a19,a20}
   \filldraw [black] (\p) circle (1.5pt);
  \foreach \x in {5pt,25pt,65pt,125pt,145pt}{
   \draw (\x,-10pt) node {$U$};
  }
  \foreach \x in {35pt,75pt,95pt,175pt,195pt}{
   \draw (\x,-10pt) node {$D$};
  }
  \foreach \x in {15pt,45pt,55pt,85pt,105pt,115pt,135pt,155pt,165pt,185pt}{
   \draw (\x,-10pt) node {$H$};
  }
 \end{scope}
 \begin{scope}[xshift=250]
  \draw[->,dashed] (0pt,0pt) -- (0,35pt);
  \draw[->,dashed] (0pt,0pt) -- (230pt,0pt);
  \draw (10pt,30pt) node {b)};
  \coordinate (a0) at (0pt,0pt);
  \coordinate (a1) at (10pt,10pt);
  \coordinate (a2) at (20pt,10pt);
  \coordinate (a3) at (30pt,20pt);
  \coordinate (a4) at (40pt,10pt);
  \coordinate (a5) at (50pt,10pt);
  \coordinate (a6) at (60pt,10pt);
  \coordinate (a7) at (70pt,20pt);
  \coordinate (a8) at (80pt,10pt);
  \coordinate (a9) at (90pt,10pt);
  \coordinate (a10) at (100pt,0pt);
  \coordinate (a11) at (110pt,10pt);
  \coordinate (a12) at (120pt,10pt);
  \coordinate (a13) at (130pt,10pt);
  \coordinate (a14) at (140pt,20pt);
  \coordinate (a15) at (150pt,20pt);
  \coordinate (a16) at (160pt,30pt);
  \coordinate (a17) at (170pt,30pt);
  \coordinate (a18) at (180pt,30pt);
  \coordinate (a19) at (190pt,20pt);
  \coordinate (a20) at (200pt,20pt);
  \coordinate (a21) at (210pt,10pt);
  \coordinate (a22) at (220pt,0pt);
  \draw (a0) -- (a1) -- (a2) -- (a3) -- (a4) -- (a5) -- (a6) -- (a7) -- (a8) -- (a9) -- (a10) -- (a11) -- (a12) -- (a13) -- (a14) -- (a15) -- (a16) -- (a17) -- (a18) -- (a19) -- (a20) -- (a21) -- (a22);
  \draw[blue, line width=1.5pt] (a10) -- (a11);
  \draw[blue, line width=1.5pt] (a21) -- (a22);
  \foreach \p in {a0,a1,a2,a3,a4,a5,a6,a7,a8,a9,a10,a11,a12,a13, a14,a15,a16,a17,a18,a19,a20,a21,a22}
   \filldraw [black] (\p) circle (1.5pt);
  \foreach \x in {5pt,25pt,65pt,135pt,155pt}{
   \draw (\x,-10pt) node {$U$};
  }
  \foreach \x in {35pt,75pt,95pt,185pt,205pt}{
   \draw (\x,-10pt) node {$D$};
  }
  \foreach \x in {15pt,45pt,55pt,85pt,115pt,125pt,145pt,165pt,175pt,195pt}{
   \draw (\x,-10pt) node {$H$};
  }
  \draw[blue] (105pt,-10pt) node {\bf\emph{U}};
  \draw[blue] (215pt,-10pt) node {\bf\emph{D}};
 \end{scope}
 \begin{scope}[yshift=-60]
  \draw[->,dashed] (0pt,0pt) -- (0,35pt);
  \draw[->,dashed] (0pt,0pt) -- (230pt,0pt);
  \draw (10pt,30pt) node {c)};
  \coordinate (a0) at (0pt,0pt);
  \coordinate (a1) at (10pt,10pt);
  \coordinate (a2) at (20pt,20pt);
  \coordinate (a3) at (30pt,30pt);
  \coordinate (a4) at (40pt,20pt);
  \coordinate (a5) at (50pt,30pt);
  \coordinate (a6) at (60pt,20pt);
  \coordinate (a7) at (70pt,10pt);
  \coordinate (a8) at (80pt,20pt);
  \coordinate (a9) at (90pt,10pt);
  \coordinate (a10) at (100pt,0pt);
  \coordinate (a11) at (110pt,10pt);
  \coordinate (a12) at (120pt,20pt);
  \coordinate (a13) at (130pt,10pt);
  \coordinate (a14) at (140pt,20pt);
  \coordinate (a15) at (150pt,20pt);
  \coordinate (a16) at (160pt,30pt);
  \coordinate (a17) at (170pt,30pt);
  \coordinate (a18) at (180pt,30pt);
  \coordinate (a19) at (190pt,20pt);
  \coordinate (a20) at (200pt,20pt);
  \coordinate (a21) at (210pt,10pt);
  \coordinate (a22) at (220pt,0pt);
  \draw (a0) -- (a1) -- (a2) -- (a3) -- (a4) -- (a5) -- (a6) -- (a7) -- (a8) -- (a9) -- (a10) -- (a11) -- (a12) -- (a13) -- (a14) -- (a15) -- (a16) -- (a17) -- (a18) -- (a19) -- (a20) -- (a21) -- (a22);
  \draw[blue, line width=1.5pt] (a1) -- (a3) -- (a4);
  \draw[blue, line width=1.5pt] (a6) -- (a7);
  \draw[blue, line width=1.5pt] (a11) -- (a12) -- (a13);
  \foreach \p in {a0,a1,a2,a3,a4,a5,a6,a7,a8,a9,a10,a11,a12,a13, a14,a15,a16,a17,a18,a19,a20,a21,a22}
   \filldraw [black] (\p) circle (1.5pt);
  \foreach \x in {5pt,45pt,75pt,105pt,135pt,155pt}{
   \draw (\x,-10pt) node {$U$};
  }
  \foreach \x in {55pt,85pt,95pt,185pt,205pt,215pt}{
   \draw (\x,-10pt) node {$D$};
  }
  \foreach \x in {145pt,165pt,175pt,195pt}{
   \draw (\x,-10pt) node {$H$};
  }
  \foreach \x in {15pt,25pt,115pt}{
   \draw[blue] (\x,-10pt) node {\bf\emph{U}};
  }
  \foreach \x in {35pt,65pt,125pt}{
   \draw[blue] (\x,-10pt) node {\bf\emph{D}};
  }
  \foreach \p in {a3,a12}{
   \filldraw [red] (\p) circle (2.5pt);
   \draw (\p) circle (2.5pt);
  }
 \end{scope}
 \begin{scope}[xshift=250, yshift=-60]
  \draw[->,dashed] (0pt,0pt) -- (0,35pt);
  \draw[->,dashed] (0pt,0pt) -- (230pt,0pt);
  \draw (10pt,30pt) node {d)};
  \coordinate (a0) at (0pt,0pt);
  \coordinate (a1) at (10pt,10pt);
  \coordinate (a2) at (20pt,20pt);
  \coordinate (a3) at (30pt,30pt);
  \coordinate (a4) at (40pt,20pt);
  \coordinate (a5) at (50pt,30pt);
  \coordinate (a6) at (60pt,20pt);
  \coordinate (a7) at (70pt,10pt);
  \coordinate (a8) at (80pt,20pt);
  \coordinate (a9) at (90pt,10pt);
  \coordinate (a10) at (100pt,0pt);
  \coordinate (a11) at (110pt,10pt);
  \coordinate (a12) at (120pt,20pt);
  \coordinate (a13) at (130pt,10pt);
  \coordinate (a14) at (140pt,20pt);
  \coordinate (a15) at (150pt,30pt);
  \coordinate (a16) at (160pt,20pt);
  \coordinate (a17) at (170pt,30pt);
  \coordinate (a18) at (180pt,30pt);
  \coordinate (a19) at (190pt,30pt);
  \coordinate (a20) at (200pt,20pt);
  \coordinate (a21) at (210pt,10pt);
  \coordinate (a22) at (220pt,0pt);
  \draw (a0) -- (a1) -- (a2) -- (a3) -- (a4) -- (a5) -- (a6) -- (a7) -- (a8) -- (a9) -- (a10) -- (a11) -- (a12) -- (a13) -- (a14) -- (a15) -- (a16) -- (a17) -- (a18) -- (a19) -- (a20) -- (a21) -- (a22);
  \draw[blue, line width=1.5pt] (a14) -- (a15) -- (a16);
  \foreach \p in {a0,a1,a2,a3,a4,a5,a6,a7,a8,a9,a10,a11,a12,a13, a14,a15,a16,a17,a18,a19,a20,a21,a22}
   \filldraw [black] (\p) circle (1.5pt);
  \foreach \x in {5pt,15pt,25pt,45pt,75pt,105pt,115pt,135pt,165pt}{
   \draw (\x,-10pt) node {$U$};
  }
  \foreach \x in {35pt,55pt,65pt,85pt,95pt,125pt,195pt,205pt,215pt}{
   \draw (\x,-10pt) node {$D$};
  }
  \foreach \x in {175pt,185pt}{
   \draw (\x,-10pt) node {$H$};
  }
  \foreach \x in {145pt}{
   \draw[blue] (\x,-10pt) node {\bf\emph{U}};
  }
  \foreach \x in {155pt}{
   \draw[blue] (\x,-10pt) node {\bf\emph{D}};
  }
  \foreach \p in {a5,a8,a15}{
   \filldraw [red] (\p) circle (2.5pt);
   \draw (\p) circle (2.5pt);
  }
 \end{scope}
 \begin{scope}[xshift=0, yshift=-130]
  \draw[->,dashed] (0pt,0pt) -- (0,45pt);
  \draw[->,dashed] (0pt,0pt) -- (230pt,0pt);
  \draw (10pt,40pt) node {e)};
  \coordinate (a0) at (0pt,0pt);
  \coordinate (a1) at (10pt,10pt);
  \coordinate (a2) at (20pt,20pt);
  \coordinate (a3) at (30pt,30pt);
  \coordinate (a4) at (40pt,20pt);
  \coordinate (a5) at (50pt,30pt);
  \coordinate (a6) at (60pt,20pt);
  \coordinate (a7) at (70pt,10pt);
  \coordinate (a8) at (80pt,20pt);
  \coordinate (a9) at (90pt,10pt);
  \coordinate (a10) at (100pt,0pt);
  \coordinate (a11) at (110pt,10pt);
  \coordinate (a12) at (120pt,20pt);
  \coordinate (a13) at (130pt,10pt);
  \coordinate (a14) at (140pt,20pt);
  \coordinate (a15) at (150pt,30pt);
  \coordinate (a16) at (160pt,20pt);
  \coordinate (a17) at (170pt,30pt);
  \coordinate (a18) at (180pt,40pt);
  \coordinate (a19) at (190pt,30pt);
  \coordinate (a20) at (200pt,20pt);
  \coordinate (a21) at (210pt,10pt);
  \coordinate (a22) at (220pt,0pt);
  \draw (a0) -- (a1) -- (a2) -- (a3) -- (a4) -- (a5) -- (a6) -- (a7) -- (a8) -- (a9) -- (a10) -- (a11) -- (a12) -- (a13) -- (a14) -- (a15) -- (a16) -- (a17) -- (a18) -- (a19) -- (a20) -- (a21) -- (a22);
  \draw[blue, line width=1.5pt] (a17) -- (a18) -- (a19);
  \foreach \p in {a0,a1,a2,a3,a4,a5,a6,a7,a8,a9,a10,a11,a12,a13, a14,a15,a16,a17,a18,a19,a20,a21,a22}
   \filldraw [black] (\p) circle (1.5pt);
  \foreach \x in {5pt,15pt,25pt,45pt,75pt,105pt,115pt,135pt,145pt,165pt}{
   \draw (\x,-10pt) node {$U$};
  }
  \foreach \x in {35pt,55pt,65pt,85pt,95pt,125pt,155pt,195pt,205pt,215pt}{
   \draw (\x,-10pt) node {$D$};
  }
  \foreach \x in {175pt}{
   \draw[blue] (\x,-10pt) node {\bf\emph{U}};
  }
  \foreach \x in {185pt}{
   \draw[blue] (\x,-10pt) node {\bf\emph{D}};
  }
  \foreach \p in {a18}{
   \filldraw [red] (\p) circle (2.5pt);
   \draw (\p) circle (2.5pt);
  }
 \end{scope}
 \begin{scope}[xshift=250, yshift=-130]
  \draw[->,dashed] (0pt,0pt) -- (0,45pt);
  \draw[->,dashed] (0pt,0pt) -- (230pt,0pt);
  \draw (10pt,40pt) node {f)};
  \coordinate (a0) at (0pt,0pt);
  \coordinate (a1) at (10pt,10pt);
  \coordinate (a2) at (20pt,20pt);
  \coordinate (a3) at (30pt,30pt);
  \coordinate (a4) at (40pt,20pt);
  \coordinate (a5) at (50pt,30pt);
  \coordinate (a6) at (60pt,20pt);
  \coordinate (a7) at (70pt,10pt);
  \coordinate (a8) at (80pt,20pt);
  \coordinate (a9) at (90pt,10pt);
  \coordinate (a10) at (100pt,0pt);
  \coordinate (a11) at (110pt,10pt);
  \coordinate (a12) at (120pt,20pt);
  \coordinate (a13) at (130pt,10pt);
  \coordinate (a14) at (140pt,20pt);
  \coordinate (a15) at (150pt,30pt);
  \coordinate (a16) at (160pt,20pt);
  \coordinate (a17) at (170pt,30pt);
  \coordinate (a18) at (180pt,40pt);
  \coordinate (a19) at (190pt,30pt);
  \coordinate (a20) at (200pt,20pt);
  \coordinate (a21) at (210pt,10pt);
  \coordinate (a22) at (220pt,0pt);
  \draw (a0) -- (a1) -- (a2) -- (a3) -- (a4) -- (a5) -- (a6) -- (a7) -- (a8) -- (a9) -- (a10) -- (a11) -- (a12) -- (a13) -- (a14) -- (a15) -- (a16) -- (a17) -- (a18) -- (a19) -- (a20) -- (a21) -- (a22);
  \foreach \p in {a0,a1,a2,a3,a4,a5,a6,a7,a8,a9,a10,a11,a12,a13, a14,a15,a16,a17,a18,a19,a20,a21,a22}
   \filldraw [black] (\p) circle (1.5pt);
  \foreach \x in {5pt,15pt,25pt,45pt,75pt,105pt,115pt,135pt,145pt,165pt,175pt}{
   \draw (\x,-10pt) node {$U$};
  }
  \foreach \x in {35pt,55pt,65pt,85pt,95pt,125pt,155pt,185pt,195pt,205pt,215pt}{
   \draw (\x,-10pt) node {$D$};
  }
 \end{scope}
\end{tikzpicture}
 \caption{Bijection. Inserted up steps and down steps are marked blue, appearing peaks are marked red.\quad a)~Initial Motzkin path.\quad b) The path after the preliminary stage.\quad c) The path after the first level iterations.\quad d) The path after the second level iterations.\quad e) The path after the third level iterations.\quad f)~Resulting Dyck path.}
\label{fig:bijection_algorithm}
\end{figure}
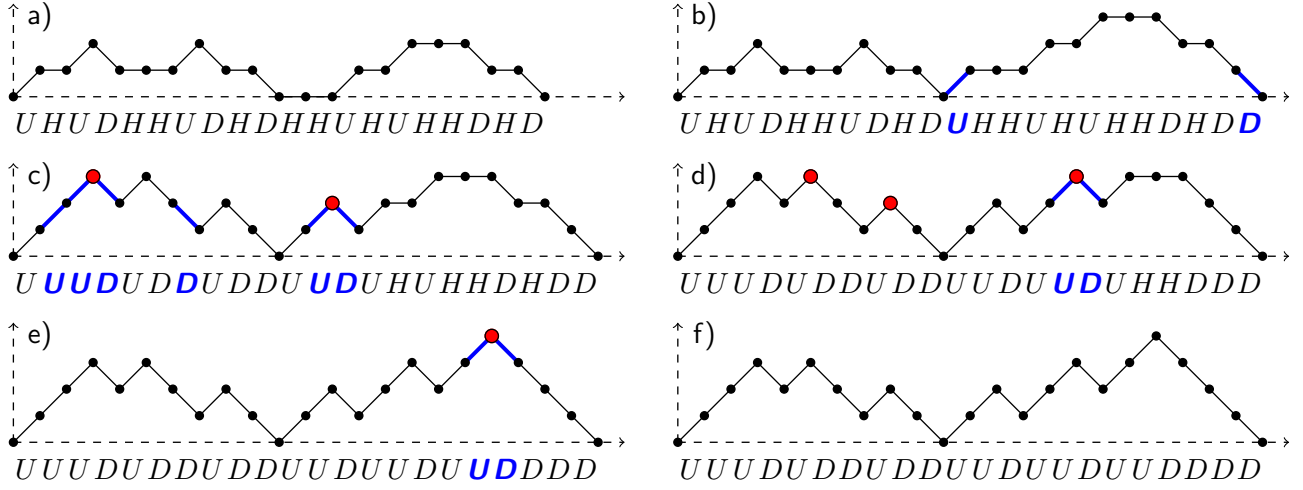

  To verify that the described algorithm realizes a bijection, it suffices to check the following two conditions.
  First, we have to show that the output of our algorithm is a Dyck path with $k+1$ peaks indeed.
  This is because the number of iterations coincides with the number of up steps in the path obtained after the preliminary stage, and each iteration produces exactly one peak.
  Second, we have to show that two different Motzkin paths at the input lead to two different Dyck paths at the output.
  This follows from the fact that a pair $(x,y)$ of a given parity is uniquely determined by the corresponding pair $(u,v)$.
  Therefore, taking into account Lemma~\ref{lem:pre_bijection}, we conclude that our algorithm provides a bijection and the number of iterations if equal to $(k+1)$.
\end{proof}

\section{Brick walk enumeration}
\label{sec:brick_walks_enumeration}

  In this section, we apply Lemma~\ref{lem:pre_bijection} to the lattice path enumeration.
  In the proof of Theorem~\ref{thm:d=4_interpretation}, we have seen how this lemma is applied to count excursions on the graph~$\G_2(1)$.
  At the same time, the presented method works in a larger setting where $\G_2(1)$ is seen as a particular example of a cone on the graph $\G_2(0)$.
  Therefore, we consider here various cones of $\G_2(0)$, including half-planes and quarter-planes (Figure~\ref{fig:cones_of_the_brick_lattice}).
  For completeness, we also discuss the cases where Lemma~\ref{lem:pre_bijection} is not needed to obtain a result.

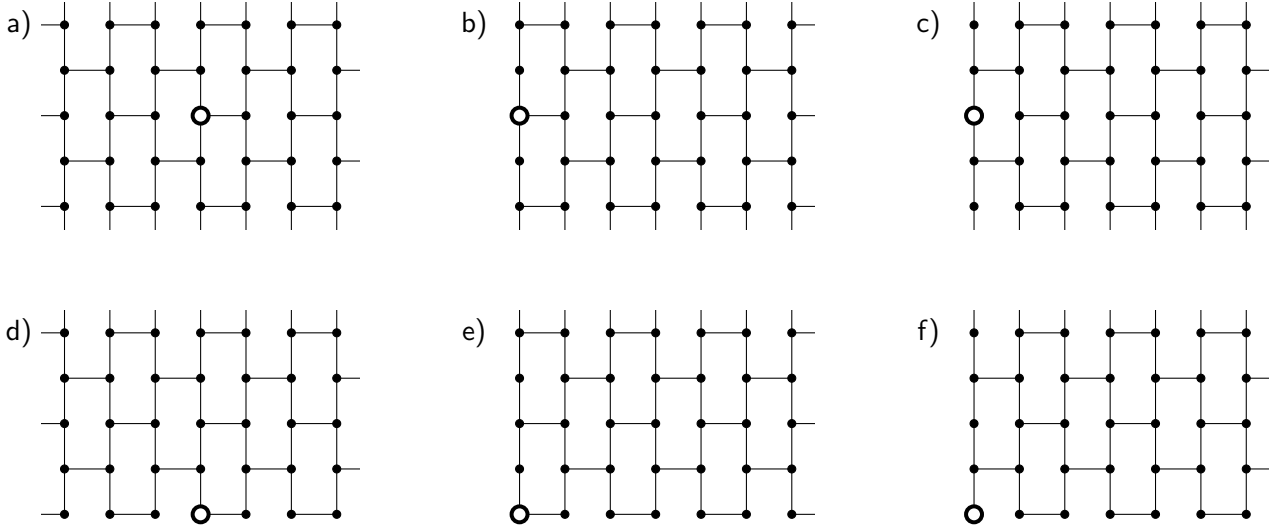
\begin{figure}[ht!]
\centering
\begin{tikzpicture}[scale = 0.6]
 \begin{scope}[xshift = -200]
  \draw (-4,2) node {a)};
  \clip (-3.5,-2.5) rectangle (3.5,2.5);
  \foreach \x in {-5,...,5}
   \draw (\x,-5) -- ++(0,10);
  \foreach \x in {-2,...,1}{
   \foreach \y in {-2,...,1}{
    \draw (2*\x,2*\y) -- ++(1,0);
    \draw (2*\x+1,2*\y+1) -- ++(1,0);
    \filldraw (2*\x,2*\y) circle (2.5pt);
    \filldraw (2*\x+1,2*\y) circle (2.5pt);
    \filldraw (2*\x+1,2*\y+1) circle (2.5pt);
    \filldraw (2*\x+2,2*\y+1) circle (2.5pt);
   }
  }
  \filldraw[white] (0,0,0) circle (5pt);
  \draw[line width=1.5pt] (0,0,0) circle (5pt);
 \end{scope}
 \begin{scope}[xshift = 0]
  \draw (-1,2) node {b)};
  \clip (-0.5,-2.5) rectangle (6.5,2.5);
  \foreach \x in {0,...,7}
   \draw (\x,-5) -- ++(0,10);
  \foreach \x in {0,...,3}{
   \foreach \y in {-2,...,2}{
    \draw (2*\x,2*\y) -- ++(1,0);
    \draw (2*\x+1,2*\y+1) -- ++(1,0);
    \filldraw (2*\x,2*\y) circle (2.5pt);
    \filldraw (2*\x+1,2*\y) circle (2.5pt);
    \filldraw (2*\x+1,2*\y+1) circle (2.5pt);
    \filldraw (2*\x+2,2*\y+1) circle (2.5pt);
   }
  }
  \foreach \y in {-1,...,2}{
   \filldraw (0,2*\y-1) circle (2.5pt);
  }
  \filldraw[white] (0,0,0) circle (5pt);
  \draw[line width=1.5pt] (0,0,0) circle (5pt);
 \end{scope}
 \begin{scope}[xshift=285]
  \draw (-1,2) node {c)};
  \clip (-0.5,-2.5) rectangle (6.5,2.5);
  \foreach \x in {0,...,7}
   \draw (\x,-5) -- ++(0,10);
  \foreach \x in {0,...,3}{
   \foreach \y in {-2,...,2}{
    \draw (2*\x,2*\y+1) -- ++(1,0);
    \draw (2*\x+1,2*\y) -- ++(1,0);
    \filldraw (2*\x,2*\y+1) circle (2.5pt);
    \filldraw (2*\x+1,2*\y+1) circle (2.5pt);
    \filldraw (2*\x+1,2*\y) circle (2.5pt);
    \filldraw (2*\x+2,2*\y) circle (2.5pt);
   }
  }
  \foreach \y in {-2,...,2}{
   \filldraw (0,2*\y) circle (2.5pt);
  }
  \filldraw[white] (0,0,0) circle (5pt);
  \draw[line width=1.5pt] (0,0,0) circle (5pt);
 \end{scope}
 \begin{scope}[xshift=-200, yshift=-250]
  \draw (-4,4) node {d)};
  \clip (-3.5,-0.5) rectangle (3.5,4.5);
  \foreach \x in {-5,...,5}
   \draw (\x,0) -- ++(0,5);
  \foreach \x in {-2,...,1}{
   \foreach \y in {0,...,2}{
    \draw (2*\x,2*\y) -- ++(1,0);
    \draw (2*\x+1,2*\y+1) -- ++(1,0);
    \filldraw (2*\x,2*\y) circle (2.5pt);
    \filldraw (2*\x+1,2*\y) circle (2.5pt);
    \filldraw (2*\x+1,2*\y+1) circle (2.5pt);
    \filldraw (2*\x+2,2*\y+1) circle (2.5pt);
   }
  }
  \filldraw[white] (0,0,0) circle (5pt);
  \draw[line width=1.5pt] (0,0,0) circle (5pt);
 \end{scope}
 \begin{scope}[yshift=-250]
  \draw (-1,4) node {e)};
  \clip (-0.5,-0.5) rectangle (6.5,4.5);
  \foreach \x in {0,...,7}
   \draw (\x,0) -- ++(0,5);
  \foreach \x in {0,...,3}{
   \foreach \y in {0,...,2}{
    \draw (2*\x,2*\y) -- ++(1,0);
    \draw (2*\x+1,2*\y+1) -- ++(1,0);
    \filldraw (2*\x,2*\y) circle (2.5pt);
    \filldraw (2*\x+1,2*\y) circle (2.5pt);
    \filldraw (2*\x+1,2*\y+1) circle (2.5pt);
    \filldraw (2*\x+2,2*\y+1) circle (2.5pt);
   }
  }
  \foreach \y in {0,...,2}{
   \filldraw (0,2*\y+1) circle (2.5pt);
  }
  \filldraw[white] (0,0,0) circle (5pt);
  \draw[line width=1.5pt] (0,0,0) circle (5pt);
 \end{scope}
 \begin{scope}[xshift=285, yshift=-250]
  \draw (-1,4) node {f)};
  \clip (-0.5,-0.5) rectangle (6.5,4.5);
  \foreach \x in {0,...,7}
   \draw (\x,0) -- ++(0,5);
  \foreach \x in {0,...,3}{
   \foreach \y in {0,...,2}{
    \draw (2*\x,2*\y+1) -- ++(1,0);
    \draw (2*\x+1,2*\y) -- ++(1,0);
    \filldraw (2*\x,2*\y+1) circle (2.5pt);
    \filldraw (2*\x+1,2*\y+1) circle (2.5pt);
    \filldraw (2*\x+1,2*\y) circle (2.5pt);
    \filldraw (2*\x+2,2*\y) circle (2.5pt);
   }
  }
  \foreach \y in {0,...,2}{
   \filldraw (0,2*\y) circle (2.5pt);
  }
  \filldraw[white] (0,0,0) circle (5pt);
  \draw[line width=1.5pt] (0,0,0) circle (5pt);
 \end{scope}
\end{tikzpicture}
 \caption{Brick lattice, different settings: a)~plane (graph $\G_2(0)$); b)~vertical half-plane; c)~reflected vertical half-plane; d)~horizontal half-plane; e)~quarter-plane; f)~reflected quarter-plane (graph $\G_2(1)$).}
\label{fig:cones_of_the_brick_lattice}
\end{figure}

\subsubsection*{Plane and horizontal half-plane}

  Let $h^{pl}_{i,j,n}$ be the number of paths on the graph $\G_2(0)$ depicted in Figure~\ref{fig:cones_of_the_brick_lattice}~a)
  that start at the origin, end at $(i,j)$, and consist of $n$ steps.
  We have seen in the proof of Corollary~\ref{cor:d=2_matrix_form} (see also~\cite{Vidakovic1994})
  that the number $h^{pl}_{0,0,2n}$ of paths returning to the origin after $2n$ steps satisfies
  \[
    h^{pl}_{0,0,2n}
      =
    \sum\limits_{k=0}^n
      \binom{2k}{k}\binom{n}{k}^2
    \, .
  \]
  Similarly, in the general case, one can establish the following relations that emphasize the dependence on the parity of the parameters:
  \[
    h^{pl}_{2i,2j,2n}
      =
    \sum\limits_{k=0}^n
      \binom{2k}{k+j}\binom{n}{k+i}\binom{n}{k-i}
    \, ,
	\qquad
    h^{pl}_{2i+1,2j+1,2n}
      =
    \sum\limits_{k=0}^n
      \binom{2k+1}{k+j+1}\binom{n}{k+i+1}\binom{n}{k-i}
    \, ,
  \]
  \[
    h^{pl}_{2i+1,2j,2n+1}
      =
    \sum\limits_{k=0}^n
      \binom{2k}{k+j}\binom{n}{k+i}\binom{n+1}{k-i}
    \, ,
	\qquad
    h^{pl}_{2i,2j+1,2n+1}
      =
    \sum\limits_{k=0}^n
      \binom{2k+1}{k+j+1}\binom{n}{k+i}\binom{n+1}{k-i+1}
    \, .
  \]
  Note that $h^{pl}_{i,j,n}=0$ if the sum of the indices $(i+j+n)$ is odd.

  In the case where we are interested in enumerating paths in the horizontal half-plane depicted in Figure~\ref{fig:cones_of_the_brick_lattice}~d),
  the corresponding number $h^{hhp}_{i,j,n}$ of paths can be obtained with the help of the reflection principle (see, for example, \cite[Chapter~10]{Bona2015}):
  \[
    h^{hhp}_{i,j,n} = h^{pl}_{i,j,n} - h^{pl}_{i,j+2,n}
    \, .
  \]

\subsubsection*{Vertical half-planes}

  Now, consider the vertical half-plane depicted in Figure~\ref{fig:cones_of_the_brick_lattice}~b) and the corresponding number of paths $h^{vhp}_{i,j,n}$ of length $n$ that start at the origin and end at the vertex $(i,j)$.
  In this case, the reflection principle does not work, so other methods are needed.
  For $i=0$, we can use Lemma~\ref{lem:pre_bijection}.
  For simplicity, assume also that $j=0$.
  As in the proof of Theorem~\ref{thm:d=4_interpretation}, we encode paths by words over the alphabet $\{U,D,R,L\}$,
  and if $k$ is the number of $U$'s (and $D$'s), then, due to Lemma~\ref{lem:pre_bijection}, the number of ways to distribute $R$'s and $L$'s in a word is $\Nar[n+1][n-k+1]=\Nar[n+1][k+1]$.
  Filling the remaining positions in the word with $U$'s and $D$'s, and taking the sum over all $k\leqslant n$, we obtain the following identity:
  \[
    h^{vhp}_{0,0,2n}
      =
    \sum\limits_{k=0}^n
      \binom{2k}{k}N(n+1,k+1)
    \, .
  \]
  More generally, the same approach allows us to obtain the following relations:
  \[
    h^{vhp}_{0,2j,2n}
      =
    \sum\limits_{k=0}^n
      \binom{2k}{k+j}N(n+1,k+1)
    \, ,
	\qquad
    h^{vhp}_{0,2j+1,2n+1}
      =
    \sum\limits_{k=0}^n
      \binom{2k+1}{k+j+1}N(n+1,k+1)
    \, .
  \]

  For the reflected vertical half-plane shown in Figure~\ref{fig:cones_of_the_brick_lattice}~c), the total number $h^{rvhp}_{i,j,k}$ of paths of length $k$ that go from the origin to the vertex $(i,j)$ can be obtained from what we have already done with the help of the relation
  \[
    h^{rvhp}_{i,j,k}
      =
    h^{vhp}_{i,j-1,k-1} + h^{vhp}_{i,j+1,k-1}
    \, .
  \]
  This relation gives us the following formulas:
  \[
    h^{rvhp}_{0,2j,0} = \delta_{j0}
    \, ,
	\quad
    h^{rvhp}_{0,2j,2n+2}
      =
    \sum\limits_{k=0}^n
      \binom{2k+2}{k+j+1}N(n+1,k+1)
    \, ,
	\quad
    h^{rvhp}_{0,2j+1,2n+1}
      =
    \sum\limits_{k=0}^n
      \binom{2k+1}{k+j+1}N(n+1,k+1)
    \, ,
  \]
  where $\delta_{ij}$ is the \emph{Kronecker delta}, that is, $\delta_{ij}=1$ for $i=j$ and $\delta_{ij}=0$ for $i\neq j$. 
  In particular, we obtain the identity $h^{rvhp}_{0,2j+1,2n+1} = h^{vhp}_{0,2j+1,2n+1}$, which can be interpreted as a path inversion.

%  By applying the kernel method, we can prove that the generating function of paths on the vertical half-plane (Figure~\ref{fig:cones_of_the_brick_lattice}) is
%  \[
%    H^{vhp}(x,y,t)
%      =
%    \sum\limits_{i=0}^\infty
%      \sum\limits_{j=-\infty}^\infty
%        \sum\limits_{k=0}^\infty
%          h^{vhp}_{i,j,k} x^i y^j t^k
%  \]
%  satisfies the following relation:
%  \[
%    H^{vhp}(x,y,t)
%      =
%    \dfrac
%    {
%      \Big(1 + t(x+Y)\Big)
%      \Big(
%        -1 + t^2(1+Y^2+2xY) +
%        \sqrt{1 - 2t^2(1+Y^2) + t^4(1-Y^2)^2}
%      \Big)
%    }
%    {2t^2xY \Big(x - t^2(x+Y)(1+xY)\Big)}
%    \, ,
%  \]
%  where $Y = y + y^{-1}$.
%  In particular, the relation
%  \[
%    H^{vhp}(0,y,t)
%      =
%    \sum\limits_{j=-\infty}^\infty
%      \sum\limits_{k=0}^\infty
%        h^{vhp}_{0,j,k}
%      =
%	\dfrac{1 + tY}{1 - t^2(1+Y^2)}
%	  \cdot
%	C\left(\dfrac{t^2Y}{1 - t^2(1+Y^2)}\right)
%    \, ,
%  \]
%  where $C(z)$ is the generating function for the Catalan numbers,
%  can be easily explained combinatorially in terms of brick walks ending on the ordinate axis.

\subsubsection*{Quarter-planes}

  Finally, consider the numbers $h^{qp}_{i,j,n}$ and $h^{rqp}_{i,j,n}$ of the paths on the quarter-planes depicted in Figure~\ref{fig:cones_of_the_brick_lattice}~e) and~f), respectively.
  In these cases, the enumeration can be obtained both by the reflection principle,
  \[
    h^{qp}_{i,j,n} = h^{vhp}_{i,j,n} - h^{vhp}_{i,j+2,n}
    	\qquad\mbox{and}\qquad
    h^{rqp}_{i,j,n} = h^{rvhp}_{i,j,n} - h^{rvhp}_{i,j+2,n}
    \, ,
  \]
  and by Lemma~\ref{lem:pre_bijection}.
  In the case of the quarter-plane, this leads to the relations
  \[
    h^{qp}_{0,2j,2n}
      =
    \sum\limits_{k=0}^n
      \dfrac{2j+1}{k+j+1}\binom{2k}{k+j}N(n+1,k+1)
    \, ,
	\qquad
    h^{qp}_{0,2j+1,2n+1}
      =
    \sum\limits_{k=0}^n
      \dfrac{2j+2}{k+j+2}\binom{2k+1}{k+j+1}N(n+1,k+1)
    \, .
  \]
  For the reflected quarter plane, that is, for the graph $\G_2(1)$, we have
  \[
    h^{rqp}_{0,2j,0} = \delta_{j0}
    \, ,
	\qquad
    h^{rqp}_{0,2j,2n+2}
      =
    \sum\limits_{k=0}^n
      \dfrac{2j+1}{k+j+2}\binom{2k+2}{k+j+1}N(n+1,k+1)
    \, ,
  \]
  \[
    h^{rqp}_{0,2j+1,2n+1}
      =
    \sum\limits_{k=0}^n
      \dfrac{2j+2}{k+j+2}\binom{2k+1}{k+j+1}N(n+1,k+1)
      =
    h^{qp}_{0,2j+1,2n+1}
    \, .
  \]
  In particular, for $n\geqslant0$ we obtain the result that we have already seen in Theorem~\ref{thm:d=4_interpretation}:
  \begin{equation}\label{eq:h^(rqp)}
    h^{rqp}_{0,0,2n+2}
      =
    \sum\limits_{k=0}^n
      C_{k+1}\Nar[n+1][k+1]
    \, .
  \end{equation}

\section{Conclusion}
\label{sec:conclusion}

  We have seen in Section~\ref{sec:introduction} that Theorem~\ref{thm:d=2_interpretation} admits a direct explanation that comes from the interpretation of the space $\mathbb{R}^2$ as a complex plane $\mathbb{C}$.
  At the same time, our proof of Theorem~\ref{thm:d=4_interpretation} is not direct: it relies not on the inner structure of random walks in $\mathbb{R}^4$, but on the known representation of moments $W_{m+1}(\nu;2n)$ in matrix form~\eqref{eq:A(nu)matrix} (the fact that comes from computing moments as integrals).
  This raises the following natural question: is it possible to obtain Theorem~\ref{thm:d=4_interpretation} directly, possibly by interpreting the space $\mathbb{R}^4$ as quaternions $\mathbb{H}$?
  Note that some of the arguments that led to Proposition~\ref{prop:d=2_pre-interpretation} still hold in this case.
  Thus, we can represent moments in the form
  \[
    W_{m+1}(1;2n)
     =
    \mathbb{E}\left(
      \Big(
        \big(A_0^{{\color{white}1}} + \ldots + A_m^{{\color{white}1}}\big)
        \big(A_0^{-1} + \ldots + A_m^{-1}\big)
      \Big)^{n}
    \right)
%    =
%    \mathbb{E}\Big(
%      (A_0^{{\color{white}1}} + \ldots + A_m^{{\color{white}1}})^{n}
%      (A_0^{-1} + \ldots + A_m^{-1})^{n}
%    \Big)
    \, ,
  \]
  where $A_k\in\mathbb{H}$ are random unit quaternions taken independently and uniformly.
  On the other hand, the evaluation of the expected value of a monomial $A=A_{i_1}^{{\color{white}1}}A_{j_1}^{-1}A_{i_2}^{{\color{white}1}}A_{j_2}^{-1}\ldots A_{i_n}^{{\color{white}1}}A_{j_n}^{-1}$ is much more complicated, since quaternions are not commutative.
  In particular, the identity $\mathbb{E}(A)\neq0$ no longer leads to the fact that $A=1$.
  For example, $\mathbb{E}(A_1^{{\color{white}1}}A_2^{-1}A_1^{{\color{white}1}}A_2^{-1})=-1/2.$

  Another question concerns higher dimensions.
  It is worth mentioning that the situation differs significantly for even and odd dimensions, so we discuss these two cases separately.
  For even dimensions $d=2(\nu+1)$ with $\nu\in\mathbb{Z}_{>1}$, it looks like the entries of the matrix $A(\nu)$ given by~\eqref{eq:A(nu)matrix} become integers after multiplication on $\binom{2\nu-1}{2}$.
  For instance, in the case where $\nu=2$, the sequence
  \[
    W_2(2;2n) = 6 \frac{(2n)!}{n!(n+2)!}
  \]
  is known as the \emph{super ballot numbers} (see~\cite{Gessel1992} and the sequence \href{https://oeis.org/A007054}{\texttt{A007054}} in~\cite{oeis}).
  For a fixed $n\in\mathbb{Z}_{\geqslant0}$, Gessel and Xin \cite{GesselXin2005} interpreted the right-hand side as the number of pairs of Dyck paths of total length $2n$ whose height difference is at most~$1$.
  This pair can also be seen as a single Dyck path of length $2n$ in which a zero-contact (that is, a point where the path touches the $x$-axis) is marked and the heights of the Dyck subpaths to the left and right of this point differ by at most~$1$.
  In general, the question is the following: is it true that the values $\binom{2\nu-1}{2}W_m(\nu;2n)$ are integers?
  If so, what could be a combinatorial interpretation of these values?
  If not, is there a multiplier that, when applied to any moment $W_m(\nu;2n)$ for a given $\nu$, turns it into an integer?

  For odd dimensions, the values are also rational, but finding a combinatorial interpretation looks much more complicated.
  Although the probability density functions $p_m(\nu;x)$ are piecewise polynomials that have an explicit form and can be computed one by one~\cite{Garcia-Pelayo2012} or with an explicit relation demanding the extraction of coefficients~\cite{BorweinSinnamon2016}, the structure of the moments remains hidden.
  For instance, we have
  \[
    W_4(1/2;s) = \dfrac{2^{s+3}(1-2^{s+2})}{(s+2)(s+3)(s+4)}
  \]
  and
  \[
    W_4(3/2;s) = \dfrac{(12)^{3}2^{s+1}\big(s^3+27s^2+230s+616+642^{s}(s^3+15s^2+62s+56)\big)}{(s+2)(s+4)(s+5)(s+6)(s+7)(s+8)(s+9)(s+10)(s+12)}
    \, ,
  \]
  see~\cite[Example~2.13]{BorweinStraubVignat2016}.
  In particular, from these expressions it seems unlikely that there is a common multiplier that would allow all the moments to be converted into integers.

  Finally, it is worth considering whether Lemma~\ref{lem:pre_bijection} can be adapted to evaluate $|\hat{\P}_{n,r,\ell}|$ for $r>\ell$.
  If the answer is yes, we could expect that the prefixes of Dyck paths will be involved in the bijection analogous to Proposition~\ref{prop:bijection}.
  This would allow us, in particular, to obtain closed formulas for $h^{vhp}_{i,j,k}$ for any non-negative integers $i$, $j$, and $k$.

%  \underline{Question.}
%  Can we say something about Hankel-total positivity of polynomials that come from matrices $A(\nu)^{m-1}$?
% Meaning that, instead of taking row sums, we could take weighted sums (elements are multiplied by $x^j$).
% Would these polynomials be Hankel-total positive for some $m$?

\section{Acknowledgements}
\label{sec:acknowledgements}

  We are very grateful to Andrew Elvey Price for his valuable comments, which showed us an elegant way to prove Proposition~\ref{prop:d=2_pre-interpretation}.
  We thank Jean-Luc Baril for the inspiration that sparked this research.
  We also thank Dmitry Mineev, Andrey Ryabichev, and everyone who participated in discussions at seminars and conferences for the fruitful exchanges.

  Sergey Kirgizov and Khaydar Nurligareev were partially supported by the French National Agency for Research (\textsc{anr}) project \textsc{PICS} ANR-22-CE48-0002.
  Michael Wallner was partially supported by the Austrian Science Fund (\textsc{fwf}) AST1535024.
  Khaydar Nurligareev and Michael Wallner were also partially supported by the \textsc{anr-fwf} project \textsc{PAnDAG} ANR-23-CE48-0014-01.

\bibliographystyle{abbrv}
\bibliography{Brick_wall}

\end{document}